\documentclass{amsart}

\usepackage{forest}
\usepackage{nicefrac}

\usetikzlibrary{shapes.geometric,decorations.markings}
\usetikzlibrary{decorations.pathreplacing}
\usetikzlibrary{fit}
\usetikzlibrary{automata}
\usetikzlibrary{positioning}

\usepackage[utf8]{inputenc}
\usepackage[T1]{fontenc}
\usepackage{textcomp}
\usepackage{amsfonts}
\usepackage{amsthm}
\usepackage{amssymb}
\usepackage{pst-node}
\usepackage[all]{xy}
\usepackage{mathtools}
\usepackage{chngcntr}
\usepackage{thmtools}
\usepackage{mathrsfs}
\usepackage{tcolorbox}
\usepackage{listings}
\usepackage{hyperref}
\usepackage{color}
\usepackage[usenames,dvipsnames]{xcolor}
\usepackage{float}
\usepackage{tikz, tikz-cd}
\usepackage{enumerate}
\usepackage[nameinlink, capitalize, noabbrev]{cleveref}
\usepackage{natbib}
\usepackage{caption}
\usetikzlibrary{intersections}
\usepackage{pgfplots}
\pgfplotsset{width=10cm,compat=1.9}
\usepgfplotslibrary{external}
\tikzset{mytext/.style={font=\small, text=black}}
\tikzset{main node/.style={circle,fill=lime!30,draw,minimum size=0.5cm,inner sep=0pt},}

\hypersetup{ colorlinks=true,linkcolor=teal,    citecolor=magenta, }

\usepackage{pst-node}
\usepackage{pst-tree}

\newcommand\Ball[1][black]{\BALL[#1]}
\def\BALL[#1](#2){\rput[t](#2){}%
        \pscircle[fillstyle=solid,fillcolor=#1!40](#2){5pt}}
        
\psset{treesep=1.3,levelsep=1.5}

\hypersetup{
    colorlinks=true,
    linkcolor=teal,
    citecolor=magenta,
    }

\numberwithin{equation}{section}

\newtheorem{Theorem}{Theorem}
\newtheorem{Conjecture}[Theorem]{Conjecture}
\newtheorem{Corollary}[Theorem]{Corollary}
\newtheorem{proposition}{Proposition}[section]
\newtheorem{lemma}[proposition]{Lemma}

\newtheorem{theorem}[proposition]{Theorem}

\theoremstyle{definition}
\newtheorem{remark}[proposition]{Remark}
\newtheorem{definition}[proposition]{Definition}

\def\NN{\mathbb N}

\def\PP{\mathbb P}

\DeclareMathOperator{\FPP}{FPP}

\DeclareMathOperator{\Gal}{Gal}
\DeclareMathOperator{\St}{St}
\DeclareMathOperator{\st}{st}
\DeclareMathOperator{\Aut}{Aut}

\DeclareMathOperator{\Sym}{Sym}

\def \<#1>{{\left\langle{#1}\right\rangle}}
\def\abs#1{\left\vert{#1}\right\vert}
\def\set#1{{\def\st{\;:\;}\left\{#1\right\}}}

\title{Fixed-point proportion of geometric iterated Galois groups}
\author{Jorge Fariña-Asategui and Santiago Radi}
\address{Jorge Fariña-Asategui: Centre for Mathematical Sciences, Lund University, 223 62 Lund, Sweden -- Department of Mathematics, University of the Basque Country UPV/EHU, 48080 Bilbao, Spain}
\email{jorge.farina\_asategui@math.lu.se}
\address{Santiago Radi: Department of Mathematics, Texas A\&M University, 77843 College Station, U.S.A.
}
\email{santiradi@tamu.edu}
\thanks{The first author is supported by the Spanish Government, grant PID2020-117281GB-I00, partly with FEDER funds. The second author is supported by Grigorchuk's Simons Foundation Grant MP-TSM-00002045 and the department of Mathematics of Texas A\&M University.}

\begin{document}

\begin{abstract}

In 1980, Odoni initiated the study of the fixed-point proportion of iterated Galois groups of polynomials motivated by prime density problems in arithmetic dynamics.

The main goal of the present paper is to completely settle the longstanding open problem of computing the fixed-point proportion of geometric iterated Galois groups of polynomials. Indeed, we confirm the well-known conjecture that Chebyshev polynomials are the only complex polynomials whose geometric iterated Galois groups have positive fixed-point proportion.  Our proof relies on methods from group theory, ergodic theory, martingale theory and complex dynamics. This result has direct applications to the proportion of periodic points of polynomials over finite fields.

The general framework developed in this paper applies more generally to rational functions over arbitrary fields and generalizes, via a unified approach, previous partial results, which have all been proved with very different methods.
\end{abstract}

\maketitle

\section{Introduction}
\label{section: introduction}

Let $L$ be a separably closed field, $K$ a subfield of $L$ and $a$ an element of $L$. Fix $K(a)^{\mathrm{sep}}$ a separable closure of $K(a)$ in $L$. Let $f\in K(x)$ be a rational function of degree $d \geq 2$. We write
$$f^n := f\circ\overset{n}{\dotsb}\circ f$$
and denote $f^{-n}(a)$ the set of preimages of $a$ under $f^n$. Let us consider the rooted tree of preimages
$$T:=\bigsqcup_{n\ge 0}f^{-n}(a).$$
As $\mathrm{Gal}(K(a)^{\mathrm{sep}}/K(a))$ acts by permutations on the $n$th level $\mathcal{L}_n:=f^{-n}(a)$ of $T$ and this action commutes with $f$, we obtain the so-called \textit{arboreal representation}
$$\rho: \mathrm{Gal}(K(a)^{\mathrm{sep}}/K(a)) \rightarrow \mathrm{Aut}(T),$$ 
where $\mathrm{Aut}(T)$ is the group of automorphisms of $T$. The kernel of the arboreal representation $\rho$ is the subgroup 
$$\mathrm{Gal}\mathrm(K(a)^{\mathrm{sep}}/K_\infty(f,a)),$$
where $K_\infty(f,a)$ is the field obtained by adjoining the elements $\bigcup_{n \geq 0} f^{-n}(a)$ to $K$. Therefore, we may define the \textit{iterated Galois group of $f$ at $a$} as 
$$G_\infty(K,f,a) := \mathrm\rho(\mathrm{Gal}(K(a)^{\mathrm{sep}}/K(a)))\cong \mathrm{Gal}(K_\infty(f,a)/K(a)).$$

Arboreal representations are a dynamic analogue of Galois representations on Tate modules of elliptic curves. A major open problem in the area is obtaining an analogue of Serre's open image theorem \cite{Serre}, i.e. finding conditions on $f$ and $a$ to ensure that the iterated Galois group has finite index in $\mathrm{Aut}(T)$. For quadratic polynomials, there has been substantial progress in this direction. For instance, a remarkable result in this vein was proved by Gratton, Nguyen and Tucker in \cite{ABC}, where, assuming the ABC conjecture, they proved that for a monic quadratic polynomial $f \in \mathbb{Z}[x]$ with wandering critical point and such that $f^n$ is irreducible over $\mathbb{Q}$ for every $n\ge 1$, the iterated Galois group $G_\infty(\mathbb{Q},f,0)$ has finite index in $\mathrm{Aut}(T)$.

The iterated Galois group $G_\infty(K,f,a)$ is a closed subgroup of $\mathrm{Aut}(T)$. Let us write $\mu_G$ for the unique Haar probability measure in a closed subgroup $G\le \mathrm{Aut}(T)$. The \textit{fixed-point proportion} of $G$ is defined as $$\mathrm{FPP}(G) := \mu_G(\{g\in G : g \text{ fixes an end in }\partial T\}),$$
where $\partial T$ is the space of ends in $T$, i.e. infinite rooted paths in $T$. In other words, the fixed-point proportion of $G$ measures the proportion of elements that fix at least one vertex at every level of $T$.

In 1980, Odoni initiated the study of the fixed-point proportion of iterated Galois groups of polynomials. Indeed, he proved in \cite{Odoni1} that
$$G_\infty(\mathbb{Q},x^2-x+1,0)=\mathrm{Aut}(T)$$
and that the fixed-point proportion of $\mathrm{Aut}(T)$ is zero; see \cite[Corollary 2.6]{AbertVirag} for a proof via Galton-Watson processes. His results were motivated by a prime density problem in arithmetic dynamics. In fact, he used the equality 
$$\mathrm{FPP}(G_\infty(\mathbb{Q},x^2-x+1,0))=0$$ 
to prove that the Dirichlet density of the set of prime numbers dividing at least one non-zero term in the Sylvester sequence, namely the sequence $w_{n+1} = 1 + w_1 \cdots w_n$ with $w_1 = 2$, is zero \cite{Odoni1}. This approach of Odoni has been greatly generalized to general number fields and rational functions by Jones and Manes in \cite{JonesManes} as follows.

Let $K$ be a number field, $a_0\in K$ and a rational function $f\in K(x)$ of degree $d\ge 2$. Then, we write 
$$P_f(a_0) := \{\mathfrak{p} \text{ prime}: v_\mathfrak{p}(f^n(a_0)) > 0  \text{ for some } n \geq 0   \text{ with } f^n(a_0) \neq 0\},$$ 
i.e. $P_f(a_0)$ is the set of primes dividing some non-zero iterate $f^n(a_0)$. If $\mathcal{D}$ denotes the Dirichlet density, Jones and Manes proved in \cite[Theorem 6.1]{JonesManes} that
\begin{align*}
\mathcal{D}(P_f(a_0)) \leq \mathrm{FPP}(G_\infty(K,f,0)).
\end{align*}
In particular, if $G_\infty(K,f,0)$ has zero fixed-point proportion, then the Dirichlet density of $P_f(a_0)$ must also be zero. This is the case, for instance, if $G_\infty(K,f,0)$ has finite index in $\mathrm{Aut}(T)$ as in the cases considered by Odoni in \cite{Odoni1, Odoni2} and Jones and Manes in \cite{JonesManes}; see \cite[Lemma~2.10]{Radi2025}.

Unfortunately, there is a few well-known obstructions for $G_\infty(K,f,0)$ to have finite index in $\mathrm{Aut}(T)$, such as $f$ having intersecting critical orbits or $f$ being post-critically finite (i.e. the set of post-critical points $P_f:=\bigcup_{n\ge 1} f^n(C_f)$ of $f$ is finite, where $C_f$ is the set of critical points of $f$); see \cite[Conjecture 3.11]{JonesArboreal}. These obstruction are a dynamic analogue of complex multiplication in elliptic curves. Therefore, a more general approach to compute the fixed-point proportion of a group is needed.

As mentioned above, \cite[Lemma~2.10]{Radi2025} implies that in order to prove that $\mathrm{FPP}(G_\infty(K,f,a))=0$, it is enough to find a finite extension of $G_\infty(K,f,a)$ with zero fixed-point proportion. In general, if we consider $t$ a transcendental element over $K$, the group $G_\infty(K,f,t)$ is expected to be a finite extension of $G_\infty(K,f,a)$ if $a$ is not a strictly post-critical point for $f$. This is the so-called \textit{specialization problem}, which has been solved in several cases; compare \cite{BGJT,BDGHT, BDGHT2, BridyTucker, JonesManes, Odoni2} and the references therein. Therefore, it is natural to study the fixed-point proportion of the overgroup $G_\infty(K,f,t)$.

We call the group $G_\infty(K,f,t)$ the \textit{arithmetic iterated Galois group} or \textit{arithmetic iterated monodromy group.} It is well-known that  
\begin{align*}
G_\infty(K,f,t) \cong \rho_{\mathrm{IMG}}( \pi_1^{\text{ét}}(\mathbb{P}^1_K \setminus P_f, x_0)),
\end{align*}
where $\rho_{\mathrm{IMG}}$ is the arboreal representation obtained via the iterated monodromy action on the tree of preimages of a point $x_0 \in \mathbb{P}^1_K \setminus P_f$  of the étale fundamental group $\pi_1^{\text{ét}}(\mathbb{P}^1_K \setminus P_f, x_0)$.

If we consider $K^{\mathrm{sep}}$ the separable closure of $K$, by Galois theory we have the short exact sequence 
\begin{align*}
1 \rightarrow G_\infty(K^{\mathrm{sep}}, f, t) \rightarrow G_\infty(K, f, t) \rightarrow \mathrm{Gal}((K^{\mathrm{sep}} \cap K_\infty(f,t))/K) \rightarrow 1.
\end{align*}
Hence, if
$$K = K_\infty(f,t)\cap K^{\mathrm{sep}},$$ 
we obtain an isomorphism and we may compute the fixed-point proportion of the Galois group  $G_\infty(K^{\mathrm{sep}},f,t)$ instead. The group $G_\infty(K^{\mathrm{sep}},f,t)$ is known as the \textit{geometric iterated Galois group} or \textit{geometric iterated monodromy group} as it is defined over a separable closed field.

Note that the geometric iterated Galois group is invariant under field extensions of the base field. Thus, if $K$ is a number field, it may be computed over~$\mathbb{C}$. If $f$ is post-critically finite, the geometric iterated Galois group $G_\infty(\mathbb{C},f,t)$ is isomorphic to the closure of the discrete iterated monodromy group $\mathrm{IMG}(f)$ in $\mathrm{Aut}(T)$, where $\mathrm{IMG}(f)$  is defined as $\rho_{\mathrm{IMG}}(\pi_1(\mathbb{P}^1_\mathbb{C}\setminus P_f,x_0))$ for the discrete fundamental group $\pi_1(\mathbb{P}^1_\mathbb{C}\setminus P_f,x_0)$. Therefore, for a number field $K$, we have
$$\mathrm{FPP}(G_\infty(K^{\mathrm{sep}}, f, t))=\mathrm{FPP}(\mathrm{IMG}(f)),$$
bringing complex dynamics to bear upon the subject. In fact, dynamical properties of complex polynomials were used by Jones in \cite{JonesAMS} to obtain results on the fixed-point proportion of iterated monodromy groups.

Iterated monodromy groups have important applications to different areas of mathematics \cite{SelfSimilar}. They were used by Bartholdi and Nekrashevych in \cite{Thurston} to solve the well-known Hubbard's twisted rabbit problem in complex dynamics; see \cite{Douady} for further details. Moreover, they provide examples of groups of intermediate growth \cite{IMG-growth,GrigorchukMilnor}, and the first examples of amenable but not subexponentially amenable groups \cite{RW2,RW1}.

As shown by the second author in \cite{Radi2025}, the fixed-point proportion of the arithmetic iterated Galois group $G_\infty(K,f,t)$ depends greatly on the base field $K$. However, a complete classification of the fixed-point proportion of geometric iterated Galois groups might still be possible.

For $d\ge 2$, the \textit{$d$th Chebyshev (complex) polynomial $T_d$} is the unique polynomial $f\in \mathbb{Z}[x]$ such that
$$f\left(x+\frac{1}{x}\right)=x^d+\frac{1}{x^d}.$$

It was proved by Jones in \cite{JonesAMS} that the fixed-point proportion of the iterated monodromy group of a Chebyshev polynomial is positive and given by
\begin{align*}
\mathrm{FPP}(\mathrm{IMG}(\pm T_d)) = \left \{ \begin{matrix} 
1/2 & \mbox{if $d$ is odd,} \\ 
1/4 & \mbox{if $d$ is even.}
\end{matrix}\right.
\end{align*}

Since Odoni started the study of the fixed-point proportion in 1980, the only known examples of polynomials $f\in \mathbb{C}[x]$ for which 
$$\mathrm{FPP}(G_\infty(\mathbb{C},f,t))>0,$$
are those linearly conjugate to $\pm T_d$. It is widely believed that these are indeed the only such examples \cite{Bridy, Juul}:
\begin{Conjecture}
\label{Conjecture: chebyshev}
    For a complex polynomial $f\in \mathbb{C}[x]$, the following are equivalent:
    \begin{enumerate}[\normalfont(i)]
    \item $f$ is linearly conjugate to $\pm T_d$;
    \item $\mathrm{FPP}(G_\infty(\mathbb{C},f,t))>0$.
    \end{enumerate}
\end{Conjecture}

However, there has only been limited partial results on \cref{Conjecture: chebyshev} in the last decades; compare \cite{Bridy, JonesAMS, JonesComp, JonesLMS, Juul}. In fact, most results on the fixed-point proportion of polynomials have been obtained for post-critically infinite maps following Odoni's original approach \cite{JonesComp, JonesLMS,Odoni1,Odoni2}. To the best of our knowledge, for post-critically finite polynomials all the known results are those obtained by Jones in \cite{JonesAMS}.

In this paper, we develop a new approach to compute the fixed-point proportion, suitable for both post-critically infinite and post-critically finite maps. As a main application, we completely characterize the fixed-point proportion of geometric iterated Galois groups of polynomials over any separably closed field, in particular verifying \cref{Conjecture: chebyshev}. To state this characterization, we need some terminology from complex dynamics, which we adapt to arbitrary separably closed fields.

We define the \textit{Thurston orbifold} of a post-critically finite rational function $f\in K(x)$ as $(\mathbb{P}_K^1,\nu_f)$ for the map $\nu_f:\mathbb{P}_K^1\to \mathbb{N}\cup \{\infty\}$, where
$$\nu_f(z):=\mathrm{lcm}\{\nu_f(w)e_f(w) : w\in f^{-n}(z) \text{ for some }n\ge 1\},$$
unless $z$ belongs to a super-attracting periodic orbit, in which case $\nu_f(z)=\infty$. As $f$ is post-critically finite, there is only finitely many points where $\nu_f(z)\ge 2$. The finite tuple $(\nu_f(z))_{z\in P_f}$ is the \textit{type} of the orbifold. 

We define the \textit{Euler characteristic} $\chi$ of the orbifold $(\mathbb{P}_K^1,\nu_f)$ as
$$\chi(\nu_f):=2-\sum_{z\in P_f}\left(1-\frac{1}{\nu_f(z)}\right).$$
A rational function $f\in K(x)$ is said to be \textit{euclidean} if $\chi(\nu_f)=0$ and \textit{hyperbolic} if $\chi(\nu_f)<0$.

Our main result shows that polynomials over $K$ with euclidean orbifolds of type $(2,2,\infty)$ are precisely the only exceptions yielding positive fixed-point proportion:

\begin{Theorem}
    \label{Theorem: main IMG intro}
    Let $K$ be a field and $f\in K[x]$ a polynomial of degree $d \geq 2$, such that either $\mathrm{char}(K)=0$ or $\mathrm{char}(K)$ does not divide the local degree of any critical point of $f$. Then, for $t$ transcendental over $K$, either
    \begin{enumerate}[\normalfont(i)]
        \item $f$ has a euclidean orbifold of type $(2,2,\infty)$ and 
        $$\mathrm{FPP}(G_\infty(K^{\mathrm{sep}}, f, t))=\left \{ \begin{matrix} 
1/2 & \mbox{if $d$ is odd,} \\ 
1/4 & \mbox{if $d$ is even;}
\end{matrix}\right.$$
        \item or $\mathrm{FPP}(G_\infty(K^{\mathrm{sep}}, f, t))=0$ otherwise.
    \end{enumerate}
\end{Theorem}

Complex polynomials with a euclidean orbifold of type $(2,2,\infty)$ are, up to a sign and linear conjugation, the Chebyshev complex polynomials; see \cite{MilnorBook}. In general, complex rational functions with a euclidean orbifold were completely classified by Douady and Hubbard in \cite{Douady}. If $K$ is separably closed, polynomials with an euclidean orbifold of type $(2,2,\infty)$ are twisted Chebyshev polynomials; see \cref{proposition: twisted chebyshev}. 

The proof of \cref{Theorem: main IMG intro} involves a mix of techniques from group theory, ergodic theory, martingale theory and complex dynamics. Before explaining the method of proof more in detail, we discuss major direct applications of \cref{Theorem: main IMG intro}.

Let $K$ be a number field and $f\in K(x)$ a rational function of degree $d \geq 2$. Given $\mathfrak{p}$ a prime ideal of good reduction for $f$, denote by $f_\mathfrak{p}$ the reduction of the coefficients of $f$ modulo $\mathfrak{p}$ and $\mathbb{F}_\mathfrak{p}$ the residue field of $\mathfrak{p}$. As $\mathbb{F}_\mathfrak{p}$ is a finite field, every point in~$\mathbb{P}^1_{\mathbb{F}_\mathfrak{p}}$ is either periodic or strictly preperiodic for $f$. We write $\mathrm{Per}(f_\mathfrak{p}, \mathbb{P}^1_{\mathbb{F}_\mathfrak{p}})\subseteq \mathbb{P}^1_{\mathbb{F}_\mathfrak{p}}$ for the set of periodic points. It was shown by Juul, Kurlberg, Madhu and Tucker in \cite{Juul} that 
\begin{align}
\label{equation: FPP and reduction of maps}
\liminf_{N(\mathfrak{p}) \rightarrow \infty} \frac{\# \mathrm{Per}(f_\mathfrak{p}, \mathbb{P}^1_{\mathbb{F}_\mathfrak{p}})}{N(\mathfrak{p}) + 1} &\le \mathrm{FPP}(G_\infty(\mathbb{C},f,t)).
\end{align}

They asked in \cite[Question 1.4]{Juul} whether one can classify all the rational functions such that the lower limit in \cref{equation: FPP and reduction of maps} is non-zero. Chebyshev polynomials are known to yield a non-zero lower limit in \cref{equation: FPP and reduction of maps} if and only if they are of prime power degree; see \cite[Example 7.2]{Juul}. Thus, \cref{Theorem: main IMG intro} answers \cite[Question~1.4]{Juul} for all polynomials:

\begin{Corollary}
    \label{Corollary: periodic points 1}
    Let $K$ be a number field and $f\in K[x]$ a polynomial of degree $d\ge 2$ and write $T_d$ for the $d$th Chebyshev polynomial. Then, 
    $$\liminf_{N(\mathfrak{p}) \rightarrow \infty} \frac{\# \mathrm{Per}(f_\mathfrak{p}, \mathbb{P}^1_{\mathbb{F}_\mathfrak{p}})}{N(\mathfrak{p}) + 1}=0$$ if and only if $f$ is not linearly conjugate to $\pm T_d$ with $d$ a prime power.
\end{Corollary}

If $K = \mathbb{F}_q$ is a finite field of characteristic $p$, we may compute the geometric iterated Galois group over $\mathbb{F}_q^{\mathrm{sep}}$. Recently, Bridy, Jones, Kelsey and Lodge proved in \cite{Bridy} that if $\mathbb{F}_q$ is a finite field of characteristic $p$ and $f\in \mathbb{F}_q(x)$ is a rational function of degree $2 \leq d < p$, then  
\begin{align}
\label{equation: FPP and finite fields}
\liminf_{n \rightarrow \infty} \frac{\# \mathrm{Per}(f, \mathbb{P}^1_{\mathbb{F}_{q^n}})}{q^n +1 } \leq \mathrm{FPP}(G_\infty(\mathbb{F}_q^{\mathrm{sep}},f,t)).
\end{align}

For a polynomial $f$, the authors in \cite{Bridy} expect the lower limit in \cref{equation: FPP and finite fields} to be zero if $f$ is not linearly conjugate to a Chebyshev polynomial. By \cref{Theorem: main IMG intro}, 
see also \cref{proposition: twisted chebyshev}, we obtain:

\begin{Corollary}
    \label{Corollary: periodic points 2}
    Let $f\in \mathbb{F}_q[x]$ be a polynomial of degree $2\le d< p$, where $p$ is the characteristic of $\mathbb{F}_q$. If $f$ is not a twisted Chebyshev polynomial, then
    $$\liminf_{n \rightarrow \infty} \frac{\# \mathrm{Per}(f, \mathbb{P}^1_{\mathbb{F}_\mathfrak{p}})}{q^n + 1}=0.$$
\end{Corollary}
 
\textcolor{teal}{Corollaries} \ref{Corollary: periodic points 1} and \ref{Corollary: periodic points 2} hold more generally for rational functions satisfying the assumptions of \cref{Theorem: mixing rational functions} below. 

Let us describe now the proof of \cref{Theorem: main IMG intro}. The \textit{fixed-point process} of the closed group $G\le \mathrm{Aut}(T)$ is the real stochastic process $\{X_n\}_{n\ge 1}$, where the random variables $X_n:(G,\mu_G)\to \mathbb{N}\cup\{ 0\}$ are given by
$$X_n(g):=\# \text{ vertices at level }n\text{ fixed by }g.$$
The martingale strategy of Jones \cite{JonesAMS, JonesComp} consists on first showing that the fixed-point process $\{X_n\}_{n\ge 1}$ is a martingale and then showing that $\lim_{n\to \infty}X_n(g) = 0$ almost surely. It is this second step for which there is not a general approach. The main novelty in the proof of \cref{Theorem: main IMG intro} is using ergodic theory to establish this second step. This ergodic theory for self-similar groups was initiated by
the first author in \cite{JorgeCyclicity} and it has already found multiple applications in the study of random elements of self-similar groups \cite{JorgeCyclicity, MarkovProcesses, JorgeSantiFPP, RandomSubgroups}.

If $T$ is the $d$-regular rooted tree for some $d\ge 2$, then $T$ may be identified with the free monoid of rank $d$. Then, for each vertex $v\in T$, we may define the operator $\mathcal{T}_v:\mathrm{Aut}(T)\to \mathrm{Aut}(T)$ given by
$$\mathcal{T}_v(g):=g|_v,$$
where $g|_v$ denotes the \textit{section of $g$ at $v$}, i.e. the restriction of $g$ to the subtree rooted at $v$.

A group $G\le \mathrm{Aut}(T)$ is said to be \textit{self-similar} if $g|_v\in G$ for every $g\in G$ and $v\in T$ and \textit{fractal} if $G$ is self-similar, acts transitively on every level of $T$ and the map $\varphi_v:\mathrm{st}_G(v)\to G$ given by $g\mapsto g|_v$ is surjective.

If $G$ is fractal and closed, it was shown by the first author in \cite[Theorem~A]{JorgeCyclicity} that $\mathcal{T}_v$ preserves the Haar measure $\mu_G$ in $G$ for every $v\in T$ and we obtain a measure-preserving dynamical system $(G,\mu_G,T,\mathcal{T})$, where $\mathcal{T}$ denotes the monoid action of $T$ on $G$ given by the operators $\mathcal{T}_v$.

As a first step towards \cref{Theorem: main IMG intro}, we introduce a new subfamily of fractal groups that we call \textit{mixing} groups; see \cref{definition: mixing group} for the precise definition of a mixing group. The term mixing is motivated by the mixing property satisfied by the associated measure-preserving dynamical system; see \cref{proposition: mixing property}. Using this mixing property, we provide our main tool to compute the fixed-point proportion of a self-similar group:

\begin{Theorem}
    \label{Theorem: main result FPP intro}
    Let $G\le \mathrm{Aut}(T)$ be a mixing group. Then $\mathrm{FPP}(G)=0$.
\end{Theorem}

\cref{Theorem: main result FPP intro} suggests that the ergodic theory initiated in \cite{JorgeCyclicity} is the right approach to study the fixed-point proportion of self-similar groups. \cref{Theorem: main result FPP intro} greatly generalizes a result of the authors in \cite[Theorem 1]{JorgeSantiFPP} for super strongly fractal groups.

In order to prove \cref{Theorem: main IMG intro}, we shall establish a dichotomy for the geometric iterated Galois group of a polynomial $f$: either $f$ has an euclidean orbifold of type $(2,2,\infty)$ or the geometric iterated Galois group $G_\infty(K^{\mathrm{sep}},f,t)$ is mixing; see \cref{lemma: mixing polynomial}. Thus, this ergodic approach singles out precisely the exceptional cases with positive fixed-point proportion. To prove this dichotomy, we generalize the notion of exceptional sets in complex dynamics.

Let $f\in \mathbb{C}(x)$ be a complex rational function. In \cite{MakarovSmirnov}, Makarov and Smirnov defined a set $\Sigma\subseteq P_f$ to be \textit{exceptional} if
$$\Sigma := f^{-1}(\Sigma) \setminus C_f.$$ 
Here, we consider a rational function $f\in K(x)$ and define  $\Upsilon\subseteq P_f$ to be \textit{critically exceptional} if
$$\Upsilon := f^{-1}(\Upsilon) \setminus ((C_f\cup P_f)\setminus \Upsilon).$$ 

Riemann-Hurwitz Theorem imposes strong constraints on the maximal size of a critically exceptional set; see \cref{lemma: almost exceptional set}. We write $\Upsilon_f$ for the maximal critically exceptional set for $f$. For a polynomial $f\in K[x]$, either $\#\Upsilon_f\le 1$ or $f$ has a euclidean orbifold of type $(2,2,\infty)$; see \cref{lemma: euclidean and exceptional}. If $f$ is a polynomial, the monodromy at infinity yields a level-transitive element and the fixed-point process of $G_\infty(K^{\mathrm{sep}},f,t)$ is a martingale; see \cref{proposition: martingale characterization}. Then, the desired dichotomy follows from the following stronger result:

\begin{Theorem}
\label{Theorem: mixing rational functions}
Let $K$ be a field and $f\in K(x)$ a rational function of degree $d \geq 2$, such that either $\mathrm{char}(K)=0$ is zero or $\mathrm{char}(K)$ does not divide the local degree of any critical point of $f$. Let us consider $t$ transcendental over $K$. Assume further that:
\begin{enumerate}[\normalfont(i)]
    \item the fixed-point process $\{X_n\}_{n\ge 1}$ of $G_\infty(K^{\mathrm{sep}},f,t)$ is a martingale;
    \item we have $\# \Upsilon_f\le 1$.
\end{enumerate}
 Then, the group  $G_\infty(K^{\mathrm{sep}},f,t)$ is mixing. 
\end{Theorem}

The martingale condition in (i) is known to hold for all polynomials and for rational functions in many cases: when the degree $d$ is prime, or $f$ has double transitive monodromy for instance; see \cite[Corolario 5.3]{Bridy}. However, it is also known that not all rational functions satisfy the martingale condition; see \cite{NonMartingale} for an example.

We expect that the approach in \cref{Theorem: mixing rational functions} may be used to further classify all rational functions for which the fixed-point proportion of their geometric iterated Galois groups is positive. In fact, we believe that $G_\infty(K,f,t)$ is always mixing for hyperbolic rational functions. As powering maps also yield mixing groups, the exceptional cases should be those with a euclidean orbifold distinct from $(\infty,\infty)$:
\begin{Conjecture}
\label{Conjecture: rational}
    Let $K$ be a field and $f\in K[x]$ a polynomial of degree $d \geq 2$, such that either $\mathrm{char}(K)=0$ or $\mathrm{char}(K)$ does not divide the local degree of any critical point of $f$. Then, for $t$ transcendental over $K$, either
    \begin{enumerate}[\normalfont(i)]
        \item $f$ has a euclidean orbifold not of type $(\infty,\infty)$ and 
        $$\mathrm{FPP}(G_\infty(K^{\mathrm{sep}}, f, t))>0$$
        \item or $\mathrm{FPP}(G_\infty(K^{\mathrm{sep}}, f, t))=0$ otherwise.
    \end{enumerate}
\end{Conjecture}

By the classification of Douady and Hubbard of complex rational functions in \cite{Douady}, in characteristic zero these are precisely Chebyshev polynomials and Lattès maps. This coincides with the widespread belief that Lattès maps are the other only complex rational functions with positive fixed-point proportion \cite{Bridy}. In fact, certain Lattès map are already known to have positive fixed-point proportion; see \cite[Theorem~2.5]{Bridy}.

\cref{Conjecture: rational} is further supported by \cref{proposition: level where one sees every pattern}, as we expect mixing to be equivalent to containing a proper projection-invariant open subgroup; see \cref{section: groups acting on trees} for the definition of projection-invariant subgroups. The rational functions whose geometric iterated Galois groups contain proper projection-invariant open subgroups have been classified by Adams in \cite{Ophelia}. They correspond precisely to quotients of strict dynamical Galois pullbacks. We plan to address the case of rational functions in a future paper.

\subsection*{\textit{\textmd{Organization}}}  

First, we introduce the new class of mixing groups in \cref{section: groups acting on trees} and we relate them to projection-invariant subgroups and the work of Adams in \cite{Ophelia}. We further develop the commutator trick in \cref{lemma: commutator trick}, as our main tool to check that a given group is mixing. In \cref{section: Dynamics}, we revise the ergodic theory for self-similar groups developed by the first author in \cite{JorgeCyclicity} and prove the key mixing property. In \cref{section: FPP}, we recall the martingale strategy of Jones and show how to use the mixing property proved in \cref{section: Dynamics} to prove \cref{Theorem: main result FPP intro}. \cref{section: Geometric iterated Galois groups of rational functions} is devoted to proving \cref{Theorem: mixing rational functions}. We introduce the key notion of critically exceptional sets and show how to use them in order to apply the commutator trick to geometric iterated Galois groups. Lastly, in \cref{section: topological polynomials} we specialize to polynomials and prove \cref{Theorem: main IMG intro}.

\subsection*{\textit{\textmd{Notation}}} Exponential notation will be used for group actions on $T$ and on its space of ends $\partial T$. Note that we consider right action on $T$ and $\partial T$.

\subsection*{Acknowledgements} 

We would like to thank Rafe Jones and Thomas J. Tucker for very helpful discussions on the fixed-point proportion of geometric iterated Galois groups. Furthermore, we would like to thank Volodymyr Nekrashevych for very helpful discussions on iterated monodromy groups. The first author would like to thank his advisors Gustavo A. Fernández-Alcober and Anitha Thillaisundaram for their constant support, while the second author thanks his advisor Rostislav Grigorchuk for his support and feedback.

\section{Groups acting on trees}
\label{section: groups acting on trees}

We first introduce the notions on self-similar groups needed in the rest of the paper. Next, we introduce the new class of mixing groups and provide the main tool to check this mixing property. Lastly, we show the relation between the mixing property and projection-invariant open subgroups of fractal groups.

\subsection{Self-similar and fractal groups}

A \textit{$d$-regular rooted tree} is a pair $(\mathscr{T}, \emptyset)$, where $\mathscr{T}$ is a tree and $\emptyset$ is a vertex in $\mathscr{T}$ such that $\deg(v) = d+1$ for every vertex $v \neq \emptyset$ and $\deg(\emptyset) = d$. The vertex $\emptyset$ is called the \textit{root} of $\mathscr{T}$.  

Given a vertex $v$, the \textit{norm} of $v$, denoted $\abs{v}$, is the distance in the tree between~$v$ and $\emptyset$. For any $n \geq 1$, the \textit{$n$th level} of the tree $\mathscr{T}$ is the set 
$$\mathcal{L}_n := \set{v \in \mathscr{T}: \abs{v} = n}.$$ 
If $1\le m \le \infty$, the \textit{$m$th truncated tree} $\mathscr{T}^m$ is the induced subtree 
$$\mathscr{T}^m:= \set{v \in \mathscr{T}: \abs{v} \le m}$$
Note that $\mathscr{T}^\infty = \mathscr{T}$.

Let $\mathrm{Aut}(\mathscr{T})$ be the group of graph automorphisms of the $d$-regular rooted tree $\mathscr{T}$ acting on $\mathscr{T}$ on the right. As the root has different degree than the other vertices, necessarily $(\emptyset)^g = \emptyset$ for every $g \in \Aut(\mathscr{T})$. Then, by preservation of adjacency, we have actions $\Aut(\mathscr{T}) \curvearrowright \mathcal{L}_n$ and $\Aut(\mathscr{T}) \curvearrowright \mathscr{T}^m$ for every $n \geq 1$ and $1 \leq m \leq \infty$. This latter action induces a surjective group homomorphism $$\pi_m: \Aut(\mathscr{T}) \rightarrow \Aut(\mathscr{T}^m).$$

Given a vertex $v \in \mathscr{T}$, we write $\mathrm{st}(v)$ for the \textit{stabilizer of the vertex $v$}. Given a natural number $n\ge 1$, we define the \textit{stabilizer of level $n$} as the pointwise stabilizer of the vertices at level $n$, i.e. $\mathrm{St}(n) = \bigcap_{v \in \mathcal{L}_n} \mathrm{st}(v)$. Clearly $\mathrm{St}(n)$ is a normal subgroup of finite index in $\mathrm{Aut}(\mathscr{T})$. The \textit{congruence topology} in $\mathrm{Aut}(\mathscr{T})$ is the profinite topology generated by the level stabilizers $\{\mathrm{St}(n)\}_{n \geq 1}$. 

If $G$ is a subgroup of $\Aut(\mathscr{T})$, then we obtain an action $G \curvearrowright \mathcal{L}_n$ for every $n \geq 1$. We say that $G$ is \textit{level-transitive} if the action $G \curvearrowright \mathcal{L}_n$ is transitive for all $n \geq 1$. We also define vertex stabilizers and level stabilizers in $G$ by restricting the ones of $\mathrm{Aut(\mathscr{T})}$, namely, $\st_G(v) := G \cap \st(v)$ and $\St_G(n) := \St(n) \cap G$ for all $v \in \mathscr{T}$ and $n \geq 1$.

Let $X := \set{1, \dots, d}$ and let $X^\NN$ be the free monoid on $X$. Denote $\emptyset_X$ the empty word in $X^\NN$. From $X^\NN$, we can construct a $d$-regular rooted tree $(T, \emptyset_X)$ as follows:
\begin{enumerate}[\normalfont(i)]
\item the vertices of $T$ are the elements in $X^\NN$;
\item if $v = x_1 \dots x_n$ is a word of length $n$ and $w = y_{1} \dots y_{n}y_{n+1}$ is a word of length $n+1$, then $v$ and $w$ are adjacent if and only if $x_i = y_i$ for all $1\le i \le n$. 
\end{enumerate}
In the remainder of the paper, $T$ will denote the tree obtained from the monoid~$X^\NN$ as above.

Let $v \in T$. The subtree $T_v$ corresponds to the subtree of $T$ whose vertices are the vertices of the form $v w$ with $w \in X^\NN$. Note that $(T_v, v)$ is also a $d$-regular rooted tree isomorphic to $T$ via the shift map 
\begin{align*}
\sigma_v: T_v &\rightarrow T \\
vw & \mapsto w.
\end{align*}

If $g \in \Aut(T)$ and $v \in T$, the restriction of $g$ to $T_v$ is a bijection $g|_{T_v}: T_v \rightarrow T_{v^g}$. Using the shift maps $\sigma_v$ and $\sigma_{v^g}$ we obtain the map $$g|_v := \sigma_v^{-1} g|_{T_x} \sigma_{v^g} \in \Aut(T),$$ called the \textit{section of $g$ at $v$}. Note that the section $g|_v$ is the unique element of $\mathrm{Aut}(\mathscr{T})$ such that
$$(vw)^g=v^gw^{g|_v}$$
for every $w\in T$.

If $1 \leq m \leq \infty$, we set $g|_v^m := \pi_m(g|_v) \in \Aut(T^m)$. The following properties are immediate to check: 
\begin{align*}
(gh)|_v^m = g|_v^m \cdot h|_{v^g}^m,\quad (g^{-1})|_v^m = (g|_{v^{g^{-1}}}^m)^{-1}\quad  \text{and} \quad g|_{uv}^m = (g|_u)|_v^m.
\end{align*}

As $T^1 = X$, we have $\Aut(T^1) = \Sym(X)$. We have an isomorphism of groups 
\begin{align*}
\psi: \Aut(T) & \rightarrow \left(  \Aut(T) \times \overset{d}{\ldots} \times \Aut(T) \right) \rtimes \Sym(X) \\
g & \mapsto (g|_1, \dots, g|_d) \pi_1(g).
\end{align*}
By an abuse of notation, we write $g = (g|_1 \dots, g|_d) \pi_1(g)$, without writing $\psi$ explicitly, to describe elements in $\Aut(T)$. 

Let $G$ be a subgroup of $\Aut(T)$. We say that $G$ is \textit{self-similar} if $g|_v \in G$ for all $v \in T$. Note that $G$ is self-similar if and only if $\psi(G)$ is a subgroup of $\left( G \times \overset{d}{\ldots} \times G \right) \rtimes \Sym(X)$. If $G$ is self-similar, we may define the map $\varphi_v^m$ as 
\begin{align*}
\varphi_v^m: \mathrm{st}_G(v) &\rightarrow \pi_m(G) \\
g &\mapsto g|_v^m,
\end{align*}
which is a well-defined group homomorphism for any $1\le m\le\infty$. We define the \textit{projection $G_v^m$ of~$G$ at $v$ of depth $m$} as
$$G_v^m:=\varphi_v^m(\mathrm{st}_G(v)).$$
When $m=\infty$, we shall drop the superscript and write simply $\varphi_v$ and $G_v$.

We say that a self-similar group $G \leq \Aut(T)$ is \textit{fractal} if $G$ is level-transitive and for all $v\in T$ we have
$$G_v=G.$$
Similarly, we say that a self-similar group $G \leq \Aut(T)$ is \textit{super strongly fractal} if~$G$ is level-transitive and for all $n \geq 1$ and $v\in \mathcal{L}_n$ we have $$\mathrm{St}_G(n)_v=G.$$
As $\St_G(n) \leq G$, super strongly fractal implies fractal.

\begin{remark}
\label{remark: fractal groups property}
Note that if $G$ is fractal, given $g \in G$ and vertices $v,w \in \mathcal{L}_n$, there exists $\widetilde{g} \in G$ such that $v^{\widetilde{g}}=w$ and $\widetilde{g}|_v = g$. This property will be used many times throughout the paper. 
\end{remark}

If $(\mathscr{T}, \emptyset)$ is a $d$-regular rooted tree, then an isomorphism of trees $\phi: \mathscr{T} \rightarrow T$ such that $\phi(\emptyset) = \emptyset_X$ is called a \textit{labeling} of $\mathscr{T}$. A labeling $\phi$ induces a group isomorphism $\rho_\phi: \Aut(\mathscr{T}) \rightarrow \Aut(T)$ given by 
$$v^{\rho_\phi(g)} := \phi(\phi^{-1}(v)^g)$$ 
for any $v \in T$.

If $G$ is a subgroup of $\Aut(\mathscr{T})$, we define self-similarity of $G$ in terms of the labeling $\phi$. Concretely, if $\phi: \mathscr{T} \rightarrow T$ is a labeling, we say that the pair $(G, \phi)$ is \textit{self-similar} if the group $\rho_\phi(G)$ is self-similar. Analogously, a pair $(G, \phi)$ is (super-strongly) fractal if $\rho_\phi(G)$ is (super-strongly) fractal.

The non-standard distinction between an arbitrary $d$-regular rooted tree $\mathscr{T}$ and the concrete $d$-regular rooted tree $T$ is motivated by \cref{section: Geometric iterated Galois groups of rational functions}. Indeed, geometric iterated Galois groups act naturally on the $d$-regular rooted tree $T_t$ of preimages of a transcendental element $t$. However, we shall obtain from this natural action a fractal action on $T$, which will allow us to apply the results on measure-preserving dynamical systems to these groups.

\subsection{Mixing groups}

Let us introduce the main class of groups in this paper:

\begin{definition}[Mixing group]
\label{definition: mixing group}
    We say that a fractal group $G\le \mathrm{Aut}(T)$ is \textit{mixing} if
    \begin{enumerate}[\normalfont(i)]
        \item for every $n\ge 1$ the level-stabilizer $\mathrm{St}_G(n)$ acts level-transitively on the subtrees rooted at level $n$;
        \item for every $n,m\ge 1$, there exists $N:=N(m)\ge 0$ such that for any $v$ with $|v|\ge n+N$ we have
        $$\mathrm{St}_G(n)_v^m=\pi_m(G).$$
    \end{enumerate}
    
We say $N:=N(m)$ is the \textit{delay constant} of $G$.
\end{definition}

Mixing groups generalize the notion of super strongly fractal groups. Indeed, if~$G$ is super strongly fractal, then its level-stabilizers satisfy condition (i) (see \cite[Lemma 3.4]{JorgeSantiFPP}) and we have
$$\mathrm{St}_G(n)_v^m=\pi_m(G)$$
for every $m\ge 1$ and every $v$ with $|v|\ge n$.

Our main tool to show that a given group is mixing is the following commutator trick:

\begin{lemma}[Commutator trick]
Let $G\le \mathrm{Aut}(T)$ be a fractal group such that for every $n\ge 1$, the level-stabilizer $\mathrm{St}_G(n)$ acts level-transitively on the subtrees rooted at level $n$. Let us consider an element $s\in G$ and assume that there exists a natural number $N\ge 1$ and element $g\in G$ satisfying that: 
\begin{enumerate}[\normalfont(i)]
\item there are vertices $u,w\in T$ with $\abs{u} \leq \abs{w} \leq N$ fixed by $g$;
\item we have $g|_u=1$ and $g|_w=s$.
\end{enumerate}
Then, for any $n\ge1$ and any $v\in T$ such that $|v|\ge n+N$ we have
$$s\in \mathrm{St}_G(n)_v.$$
In particular, if the above holds for any element $s$ in a generating set of $G$, then $G$ is mixing with delay constant $N$.
\label{lemma: commutator trick}
\end{lemma}

\begin{proof}
Let $n \geq 1$ and $v \in T$ such that $\abs{v} \geq n + N$. Note that as $\abs{v} - \abs{w} \geq n$, we can find $w'$ in $T$ such that $\abs{v} = \abs{w' w}$. Since $G$ is fractal, there exists $\widetilde{g} \in \st_G(w')$ such that $$\widetilde{g}|_{w'} = g.$$ Note also that, as all the sections of $1$ are $1$, we may assume that $\abs{u} = \abs{w}$.

By assumption $\St_G(\abs{w'})$ acts level-transitively on the subtree rooted at $w'$. Thus, there exists $h \in \St_G(\abs{w'})$ such that $$(w'u)^h = w'w.$$ 
As $\St_G(\abs{w'})$ is normal in $G$ and $h \in \St_G(\abs{w'})$, we get
$$[h, \widetilde{g}] \in \St_G(\abs{w'}).$$
Moreover, $[h,\widetilde{g}] \in \st_G(w'w)$ and using the properties of sections, 
\begin{align*}
[h,\widetilde{g}]|_{w'w} = (h|_{w'u})^{-1} \cdot (\widetilde{g}|_{w'u})^{-1} \cdot h|_{w'u} \cdot \widetilde{g}|_{w'w} = g|_w =s.
\end{align*}

Lastly, by \cref{remark: fractal groups property}, we can conjugate $[h, \widetilde{g}]$ by an element $k$ in $G$ such that $(w'w)^k = v$ and $k|_{w'w} = 1$. Then 
\begin{equation*}
s \in \St_G(\abs{w'})_v \leq \St_G(n)_v.
\end{equation*}
as $\abs{w'} \geq n$.
\end{proof}

 As we shall see in \cref{section: Geometric iterated Galois groups of rational functions}, \cref{lemma: commutator trick} works surprisingly well for geometric iterated Galois groups.

\subsection{Connection to projection-invariant subgroups}

We say that a subgroup $K\le \mathrm{Aut}(T)$ is \textit{projection-invariant} if $K_v\le K$ for any $v\in T$. Note that it is enough to check the containment $K_v\le K$ with $v\in \mathcal{L}_1$.

\begin{remark}
A word of caution here. The notion of being projection-invariant and self-similar are not equivalent. Let us consider the group $G\le \mathrm{Aut}(T)$, generated by the involutions $a$ and $b$, which are given by the wreath recursions
$$a=(1,1)(1\,2)\quad \text{and}\quad b:=(a,b).$$
Then, the subgroup $H:=\langle ba\rangle$ is projection-invariant but not self-similar. Indeed
$$ba=(a,b)(1\,2)\quad \text{and}\quad(ba)^2=(ab,ba)=((ba)^{-1},ba).$$
Thus $\mathrm{st}_{H}(v)=H$ for any $v\in \mathcal{L}_1$ and $H$ is projection-invariant, but the two sections of $ba$ at the first level are $a$ and $b$ respectively, so $H$ is not self-similar. 
\end{remark}

Self-similar subgroups of both discrete and profinite self-similar groups have been studied by the first author in \cite{JorgeCyclicity}. Projection-invariant subgroups have appeared (under the name of self-similar subgroups!) recently in the work of Adams in \cite{Ophelia} in the context of geometric iterated Galois groups. Adams showed in \cite[Theorem~17]{Ophelia} that projection-invariant normal open subgroups of geometric iterated Galois groups correspond to Galois dynamical pullbacks.  We now establish an explicit connection between the mixing property and projection-invariant open subgroups, which connects our \cref{Theorem: main IMG intro} to Adams results in \cite[Theorem 17]{Ophelia}.

Let $G\le\mathrm{Aut}(T)$, and let us define the subgroup $K_G\le \mathrm{Aut}(T)$ as
$$K_G:=\bigcap_{v\in \gamma} \mathrm{St}_G(|v|)_v,$$
where $\gamma\in \partial T$, i.e. an infinite rooted path in $T$.

If $G$ is fractal, the subgroup  $K_G$ does not depend on the chosen path $\gamma$ by \cref{remark: fractal groups property} and both~$K_G$ and all the projections $(K_G)_v$ are normal in $G$ as $\varphi_v$ is surjective.

We say that $G$ is \textit{virtually super strongly fractal} if $G$ is fractal, $K_G$ is of finite index in $G$ and the level-stabilizer $\mathrm{St}_G(n)$ acts level-transitively on the subtrees rooted at level $n$ for every $n\ge 1$. Note that if $G$ is super strongly fractal, then $K_G=G$, which motivates the term virtually super strongly fractal. 

If $G$ is closed, then $K_G$ is always closed. Let us assume in the following that~$G$ is closed in $\mathrm{Aut}(T)$.

If $G$ is virtually super strongly fractal, then we obtain a characterization of the mixing property in terms of projection-invariant open subgroups:

\begin{proposition}
\label{proposition: level where one sees every pattern}
Let $G\le \mathrm{Aut}(T)$ be a virtually super strongly fractal closed subgroup. Then, the following are equivalent:
\begin{enumerate}[\normalfont(i)]
\item $G$ is mixing;
\item $G$ does not have any proper projection-invariant open subgroup.
\end{enumerate}
\end{proposition}
\begin{proof}
First, we prove (i)$\implies$(ii) (note that we do not use virtual super strong fractality for this implication). Let $K\le G$ be any projection-invariant open subgroup. Then $K\ge \mathrm{St}_G(m)$ for some $m\ge 1$. Since $G$ is mixing, there exists some vertex $v\in T$, such that
$$K_v^m=\pi_m(G).$$
As $K$ is projection invariant, this implies that $\pi_m(K)=\pi_m(G)$. As $K$ is open and $K\ge \mathrm{St}_G(m)$, we further deduce that $K=G$.

We now prove (ii)$\implies$(i). By assumption, the subgroup $K_G$ is open since it is closed and of finite index in $G$. Thus, it contains a level-stabilizer $\mathrm{St}_G(k)$ for some $k \geq 1$. Fix $x \in X$. We claim that there exists an $r\ge 0$ and a collection of subgroups $K_G:=K_0,K_1,\dotsc,K_r=G$ with the following properties:
\begin{enumerate}[\indent(1)]
\item $\abs{G:K_{i+1}} < \abs{G:K_i} < \infty$ for all $0\le i \le r-1$;
\item for all $u \in T$ with $\abs{u} \geq k$ , we have $K_i \leq (K_G)_{ux^i}$.
\end{enumerate}
We construct this collection of subgroups by induction. For $i = 0$, we have $K_0 := K_G$, and by definition
$$K_G \leq \St_G(\abs{u})_u \leq \St_G(k)_u \leq (K_G)_u.$$ 
Furthermore, $K_G$ has finite index in $G$ by assumption. 

Let us now define $K_{i+1}$ from $K_i$. As $K_i$ is open, it cannot be projection invariant. Thus, there exists $y \in \mathcal{L}_1$ and $h \in G$ such that $$h \in (K_i)_y \setminus K_i.$$ By \cref{remark: fractal groups property}, we can conjugate $K_i$ by an element $g \in G$ such that $y^g = x$ and $g|_y = g$, obtaining that $$g^{-1}hg \in (g^{-1} K_i g)_x \setminus (g^{-1} K_i g).$$
We define 
$$K_{i+1} := g^{-1} \langle K_i, h\rangle g.$$
By assumption $\abs{G:K_{i+1}} < \abs{G:K_i}$, so $K_{i+1}$ satisfies condition (1). Now, as $K_i$ satisfies condition (2), we have
\begin{align*}
g^{-1} K_i g \leq g^{-1} (K_G)_{ux^i} g = (K_G)_{ux^i}
\end{align*}
for all $u\in T$ with $\abs{u} \geq k$, and projecting to $x$, we further get 
$$g^{-1} h g \in (g^{-1} K_i g)_x \leq (K_G)_{ux^{i+1}}.$$
Similarly, applying (2) to the vertex $ux$ we further get
$$g^{-1} K_i g \leq (K_G)_{ux^{i+1}}.$$ 
Then, $K_{i+1} \leq (K_G)_{u x^{i+1}}$ and $K_{i+1}$ satisfies condition (2) too. This implies that $G \leq (K_G)_{ux^r}$, and by \cref{remark: fractal groups property}, we get
$$(K_G)_{u'} \geq G$$ 
for any $u' \in T$ such that $\abs{u'} \geq k+r$. Hence $G$ is mixing. Indeed, let $n,m \geq 1$ and set $N := k+r$. Then, if $v \in T$ such that $\abs{v} \geq n + N$, writing $v = v_1 v_2$ with $\abs{v_1} = n$ we conclude
\begin{equation*}
(\St_G(n))^m_v = (\St_G(\abs{v_1}))^m_{v_1v_2} \geq (K_G)_{v_2}^m \geq \pi_m(G) \qedhere
\end{equation*}

\end{proof}

\section{Dynamical systems associated to mixing groups}
\label{section: Dynamics}

In this section, we first revise the measure-preserving dynamical systems associated to fractal groups introduced by the first author in \cite{JorgeCyclicity}. Then, we study the mixing properties of these measure-preserving dynamical systems, extending previous results for super strongly fractal groups in \cite{JorgeCyclicity} and motivating the definition of the new class of mixing groups.

\subsection{Dynamical systems associated to fractal groups}

Let $(X,\mu)$ be a probability space and $S$ any monoid. A monoid action $\mathcal{S}$ of $S$ on $(X,\mu)$ is said to be \textit{measure-preserving} if for any measurable subset $Y\subseteq X$ and any $s\in S$ we have $\mu(\mathcal{S}_s^{-1}(Y))=\mu(Y)$, where $\mathcal{S}_s$ is the operator associated to the action of $s$. The tuple $(X,\mu,S,\mathcal{S})$, where $\mathcal{S}$ is measure-preserving, is called a \textit{measure-preserving dynamical system}.

The first author introduced a natural way to associate a measure-preserving dynamical system to a fractal profinite group in \cite{JorgeCyclicity}. Let $G\le \mathrm{Aut}(T)$ be a fractal closed subgroup and let $\mu_G$ be the unique Haar (Borel) probability measure in $G$. The Borel algebra of $G$ is generated by the so-called cone sets. For any $n\ge 1$ and $A\subseteq \pi_n(G)$, we define the \textit{cone set} $C_A$ as the fiber
$$C_A:=\pi_n^{-1}(A).$$
Furthermore, note that $C_A=\bigsqcup_{a\in A} g_a\mathrm{St}_G(n)$ with each $g_a\in G$ and  
$$\mu_G(C_A)=\sum_{a\in A} \mu_G(g_a\mathrm{St}_G(n))=\#A\cdot \mu_G(\mathrm{St}_G(n))=\frac{\#A}{|\pi_n(G)|}.$$
This standard fact will be used throughout the paper. 

We may view the $d$-regular rooted tree $T$ as the free monoid of rank $d$. What is more, we define a free monoid action $\mathcal{T}$ of $T$ on $(G,\mu_G)$ via sections, i.e.
$$\mathcal{T}_v(g):=g|_v$$
for every $v\in T$ and $g\in G$. It was proved by the first author in \cite{JorgeCyclicity} that this monoid action $\mathcal{T}$ is measure-preserving:

\begin{theorem}[{{see \cite[Theorem A]{JorgeCyclicity}}}]
    \label{theorem: fractal is dyn sys}
    Let $G\le \mathrm{Aut}(T)$ be a fractal profinite group. Then $(G,\mu_G,T,\mathcal{T})$ is a measure-preserving dynamical system.
\end{theorem}

Furthermore, it was shown in \cite[Theorem A]{JorgeCyclicity} that if $G$ is further assumed to be super strongly fractal, then the measure-preserving dynamical system $(G,\mu_G,T,\mathcal{T})$ satisfies a strong mixing property. The goal of the remainder of the section is to show that the class of mixing groups is precisely the subclass of fractal groups whose associated measure-preserving dynamical systems satisfy a (weaker) mixing condition.

\subsection{A pseudomixing property}

The next proposition shows that the dynamical system associated to a mixing group satisfies a pseudomixing property, which justifies the term \textit{mixing}:

\begin{proposition}[Pseudomixing]
\label{proposition: mixing property}
Let $G\le \mathrm{Aut}(T)$ be a mixing group with delay constant $N$. Then, for every $n,m\ge 1$, every vertex $v$ with $|v|\ge n+ N$ and every $A\subseteq \pi_n(G)$ and $B\subseteq \pi_m(G)$, we have
$$\mu_G(C_A\cap \mathcal{T}_v^{-1}(C_B)\cap \mathcal{T}_u^{-1}(\mathrm{st}_G(w)))=\mu_G(C_A)\cdot \mu_G(C_B)\cdot \mu_G(\mathrm{st}_G(w)),$$
where $u$ is at level $n$ and $v=uw$. In particular
$$\mu_G(C_A\cap \mathcal{T}_v^{-1}(C_B)\mid \mathcal{T}_u^{-1}(\mathrm{st}_G(w)))=\mu_G(C_A)\cdot \mu_G(C_B).$$
\end{proposition}

\begin{proof}
Let $a \in A$ and $b \in B$. Then 
\begin{align*}
C_a\cap \mathcal{T}_v^{-1}(C_b)\cap \mathcal{T}_u^{-1}(\st_G(w)) = \set{g \in G: g|_\emptyset^n = a, \,\, g|_v^m = b, \,\, w^{g|_u} = w}.
\end{align*}

As $G$ is mixing, the set above is non-empty. Indeed, let $h_3 \in G$ such that $\pi_n(h_3) = a$ and let $w' \in \mathcal{L}_{\abs{w}}$ such that $(w')^{h_3|_u} = w$. As $\St_G(n)$ acts level-transitively below $u$, there exists $h_2 \in \St_G(n)$ such that $(uw)^{h_2} = uw'$. Moreover, since $G$ is mixing, there exists $h_1 \in \St_G(n) \cap \st_G(v)$ such that 
$$h_1|_{v}^m = b \cdot (h_2|_{v}^m)^{-1} \cdot (h_3|_{uw'}^m)^{-1}.$$
Thus, 
$\pi_n(h_1 h_2 h_3) = a$, the section $(h_1 h_2 h_3)|_u$ fixes $w$ and
$$h_1 h_2 h_3|_v^m = (h_1 |_v^m) \cdot (h_2|_v^m) \cdot (h_3|_{uw'}^m) = b.$$  
Therefore, we may express 
$$C_A\cap \mathcal{T}_{v}^{-1}(C_B)\cap \mathcal{T}_u^{-1}(\mathrm{st}_G(w))$$
as the disjoint union of $(\#A)\cdot (\#B)$ cosets of $\ker(\varphi_v^m|_{\mathrm{St}_G(n)})$. By the translation-invariance of $\mu_G$, we get
\begin{align*}
\mu_G(C_A\cap \mathcal{T}_v^{-1}(C_B)\cap \mathcal{T}_u^{-1}(\st_G(w))) =
(\# A) \cdot (\#B) \cdot \mu_G(\ker(\varphi_v^m|_{\St_G(n)})).
\end{align*}

To compute the measure of the kernel, we first observe that $\varphi_{v}^m|_{\St_G(n)}$ induces a group homomorphism  $$\widehat{\varphi}: \pi_{\abs{v} + m}(\St_G(n) \cap \st_G(v)) \to \pi_m(G)$$ and hence
\begin{align*}
\mu_G(\ker(\varphi_v^m|_{\St_G(n)})) = \frac{\abs{\ker(\widehat{\varphi})}}{\abs{G: \St_G(\abs{v} + m)}}.
\end{align*}
Now, as $G$ is mixing, the map $\widehat{\varphi}$ is surjective and by the first and third isomorphism theorems we obtain that
\begin{align*}
\mu_G(\ker(\varphi_v^m|_{\St_G(n)})) & = \frac{1}{\abs{\pi_m(G)}} \cdot \frac{\abs{\pi_{\abs{v} + m}(\St_G(n) \cap \st_G(v))}}{\abs{G: \St_G(\abs{v} + m)}} \\
 & = \frac{1}{\abs{\pi_m(G)}} \cdot \frac{\abs{\St_G(n) \cap \st_G(v) : \St_G(\abs{v} + m)}}{\abs{G: \St_G(\abs{v} + m)}}\\
 &=\frac{1}{|\pi_m(G)|}\cdot \frac{1}{|G:\mathrm{St}_G(n)\cap \mathrm{st}_G(v)|}\\
 &=\frac{1}{|\pi_m(G)|} \cdot \frac{1}{|\pi_n(G)|} \cdot \frac{1}{|\mathrm{St}_G(n):\mathrm{St}_G(n)\cap \mathrm{st}_G(v)|}.
\end{align*}
As $\St_G(n)$ acts level-transitively on the subtree rooted at $u$, the orbit-stabilizer theorem yields
$$\abs{\St_G(n): \St_G(n) \cap \st_G(v)}^{-1} = d^{-\abs{w}} = \mu_G(\st_G(w)).$$ 
Combining everything, 
\begin{align*}
\mu_G(C_A\cap \mathcal{T}_v^{-1}(C_B)\cap \mathcal{T}_u^{-1}(\st_G(w))) & =
(\# A) \cdot (\#B) \cdot \mu_G(\ker(\varphi_{v}^m|_{\St_G(n)})) \\
& = \frac{\#A}{\abs{\pi_n(G)}} \cdot \frac{\#B}{\abs{\pi_m(G)}} \cdot \mu_G(\st_G(w))  \\
& = \mu_G(C_A) \cdot \mu_G(C_B) \cdot \mu_G(\st_G(w)).
\end{align*}

Finally, by \cref{theorem: fractal is dyn sys},
\begin{align*}
\mu_G(C_A \cap \mathcal{T}_v^{-1}(C_B) \mid \mathcal{T}_u^{-1}(\st_G(w))) & = \frac{\mu_G(C_A \cap \mathcal{T}_v^{-1}(C_B) \cap \mathcal{T}_u^{-1}(\st(w)))}{\mu_G(\mathcal{T}_u^{-1}(\st_G(w)))} \\ 
& = \frac{\mu_G(C_A) \cdot \mu_G(C_B) \cdot \mu_G(\st_G(w))}{\mu_G(\st_G(w))} \\
& = \mu_G(C_A) \cdot \mu_G(C_B) \qedhere
\end{align*}
\end{proof}

This pseudomixing property is precisely what we want when studying the fixed-point proportion of a self-similar group. In fact, the fixed-point proportion concerns the set of elements in the group fixing an end. Thus, it is conditioned on this set of elements precisely when we may apply the mixing property in \cref{proposition: mixing property}.

\section{Fixed-point proportion and the martingale strategy}
\label{section: FPP}

The main goal of the section is to prove \cref{Theorem: main result FPP intro}, i.e. that mixing groups have zero fixed-point proportion. First, we review the martingale strategy of Jones and introduce the fixed-point process and the fixed-point proportion of a self-similar group. Then, we use the mixing properties of the associated dynamical systems obtained in \cref{section: Dynamics} to prove \cref{Theorem: main result FPP intro}.

\subsection{Fixed-point processes, martingales and the fixed-point proportion}

Let $\mathscr{T}$ be a $d$-regular rooted tree, $G\le \mathrm{Aut}(\mathscr{T})$ a closed subgroup and $\mu_G$ its unique Haar probability measure. We define the \textit{fixed-point proportion} $\mathrm{FPP}(G)$ of $G$ as
\begin{align*}
\mathrm{FPP}(G):&=\mu_G(F_G),
\end{align*}
where $F_G\subset G$ is the closed Borel subset of $G$ given by
$$F_G:=\{g\in G : g \text{ fixes an end in }\partial \mathscr{T}\}.$$

\begin{remark}
\label{remark: FPP invariant under relabeling}
Note that if $T$ is the $d$-regular rooted tree obtained from $X^\NN$ and $\phi: \mathscr{T} \rightarrow T$ is any labeling for $\mathscr{T}$, then $\FPP(G) = \FPP(\rho_\phi(G))$. Therefore, for the remainder of \cref{section: FPP}, we may assume that $G \leq \Aut(T)$.
\end{remark}

Let $G\le \mathrm{Aut}(T)$ be a closed subgroup. The \textit{fixed-point process} of $G$ is the real stochastic process $\{X_n\}_{n\ge 1}$, where the random variables $X_n:(G,\mu_G)\to \mathbb{N}\cup\{ 0\}$ are given by
$$X_n(g):=\# \text{ vertices at level }n\text{ fixed by }g.$$

A key insight of Jones was to realize that the fixed-point process of self-similar groups is in many instances a martingale \cite{JonesComp}. Recall that a real stochastic process $\{Y_n\}_{n\ge 1}$ defined over a probability space $(X,\mu)$ is a \textit{martingale} if for all $n\ge 1$ we have
\begin{enumerate}[\normalfont(i)]
\item $\mathbb{E}(Y_n)<\infty$;
\item $\mathbb{E}(Y_{n+1}\mid Y_1=t_1,\dotsc, Y_n=t_n)=t_{n}$ for every $t_1,\dotsc, t_n\in \mathbb{R}$ such that $\mu(Y_1=t_1,\dotsc, Y_n=t_n)>0$.
\end{enumerate}

Non-negative martingales converge almost surely:

\begin{theorem}[Martingale convergence]
Let $\{Y_n\}_{n\ge 1}$ be a non-negative martingale over a probability space $(X,\mu)$ with $\mathrm{E}(Y_1)<\infty$. Then
$$\lim_{n\to\infty} Y_n(x)$$
exists for $x\in X$ almost surely and with finite value.
\end{theorem}

If the fixed-point process $\{X_n\}_{n\ge 1}$ of $G$ is a martingale, then $\{X_n\}_{n\ge 1}$ is eventually constant for almost all $g \in G$. Indeed, as $X_n(g)$ is a non-negative integer for every $n\ge 1$ and $g\in G$, the fixed-point process must stabilize after only finitely many steps. 

Those closed subgroups whose fixed-point process is a martingale have been completely characterized:

\begin{proposition}[{see {\cite[Theorem 5.3]{Bridy}} and \cite{JonesLMS}}]
\label{proposition: martingale characterization}
    Let $G\le \mathrm{Aut}(T)$ be a closed subgroup. Then, its fixed-point process is a martingale if and only if for every $n\ge 1$, the subgroup $\mathrm{St}_G(n)$ acts level-transitively on the subtrees rooted at level $n$.
\end{proposition}

Here, we are interested in providing sufficient conditions for the fixed-point proportion of a group to be zero. Equivalently, we want sufficient conditions for the fixed-point process of a group to converge to zero (to become eventually zero under the martingale assumption) almost surely. We shall use the martingale strategy developed by Jones in \cite{JonesComp}; see \cite[Lemma 3.2]{JorgeSantiFPP} for a detailed proof:

\begin{lemma}[Martingale strategy]
\label{lemma: Jones strategy}
Let $G\le \mathrm{Aut}(T)$ be a closed subgroup. Assume that:
\begin{enumerate}[\normalfont(i)]
\item the fixed-point process $\{X_n\}_{n\ge 1}$ of $G$ is a martingale;
\item for any $r>0$ there exists $\epsilon:=\epsilon(r)$ and $m:=m(r)$ such that for infinitely many $n\ge 1$ we have $\mu_G(X_{n+m} = r \mid  X_n = r) \leq 1 - \epsilon$.
\end{enumerate}
Then, the fixed-point process of $G$ is eventually zero almost surely.
\end{lemma}

\subsection{Proof of the main result}

We conclude the section by proving \cref{Theorem: main result FPP intro}. The proof builds up on the use of the mixing property of super strongly fractal groups developed by the authors in \cite{JorgeSantiFPP}. Indeed, we prove that $$\mu_G(X_{n+N}>r\mid X_n=r)\ge \epsilon>0$$
for any $n\ge 1$, i.e. that the fixed-point process cannot converge unless $r=0$. This is done by using the mixing property in \cref{proposition: mixing property} to increase the number of fixed points at any level deep enough in the tree, always with a fixed probability.

Before proving \cref{Theorem: main result FPP intro}, we restate it here for the convenience of the reader:

\begin{theorem}
    \label{theorem: main result FPP}
    Let $G\le \mathrm{Aut}(T)$ be a mixing group. Then $\mathrm{FPP}(G)=0$.
\end{theorem}
\begin{proof}

By the definition of mixing, the level stabilizers satisfy the assumption in \cref{proposition: martingale characterization}, so the fixed-point process of $G$ is a martingale. 

Let us now show that condition (ii) of \cref{lemma: Jones strategy} holds. Let $r\ge 1$ and set $m:=\lceil \log_d r\rceil$. Let us further consider the delay constant $N:=N(m)$ of~$G$. It is enough to show that there exists $\epsilon:=\epsilon(r)>0$ only depending on~$r$ such that
$$\mu_G(X_{n+(N+m)} > r\mid  X_n = r)\ge \epsilon$$
for every $n\ge 1$. Indeed,
$$\mu_G(X_{n+(N+m)} = r\mid  X_n = r)\le 1-\mu_G(X_{n+(N+m)} > r\mid  X_n = r)\le 1-\epsilon$$
for every $n\ge 1$.

We define the set $A_{n,r}\subseteq \pi_n(G)$ as
$$A_{n,r}:=\{a\in \pi_n(G):X_n(a)=r\}.$$
Then 
\begin{align*}
\mu_G(X_{n+(N+m)} > r\mid  X_n = r ) &= \frac{\mu_G(X_{n+(N+m)} > r \cap X_n = r )}{\mu_G(X_n = r)}\\
&=\frac{ \mu_G(X_{n+(N+m)} > r \cap C_{A_{n,r}} )}{\mu_G(C_{A_{n,r}})}\\
&= \frac{\sum_{a \in A_{n,r}} \mu_G(X_{n+(N+m)} > r \cap C_a )}{\mu_G(C_{A_{n,r}})}.
\end{align*}
As $r\ge 1$, each element $a$ fixes at least one vertex $u_a$ at level $n$. Let $v_a$ be a vertex~$N$ levels below $u_a$ for each $a\in A_{n,r}$. Then 
\begin{align*}
\mu_G(X_{n+(N+m)} > r\mid  X_n = r) \geq \frac{\sum_{a \in A_{n,r}} \mu_G(C_a\cap \mathcal{T}_{v_a}^{-1}(\mathrm{St}_G(m))\cap \mathcal{T}_{u_a}^{-1}(\mathrm{st}_G(v_a)))}{\mu_G(C_{A_{n,r}})}.
\end{align*}
Lastly, as $G$ is mixing and $|v_a|= n+N$, we may apply the mixing condition in \cref{proposition: mixing property} with $C_B=\mathrm{St}_G(m)$ to obtain 
\begin{align*}
\mu_G(X_{n+(N+m)} &> r\mid  X_n = r) \\
&\geq \frac{\sum_{a \in A_{n,r}} \mu_G(C_a\cap \mathcal{T}_{v_a}^{-1}(\mathrm{St}_G(m))\cap \mathcal{T}_{u_a}^{-1}(\mathrm{st}_G(v_a)))}{\mu_G(C_{A_{n,r}})}\\
&=\frac{\sum_{a \in A_{n,r}} \mu_G(C_a)\cdot \mu_G(\mathrm{St}_G(m))\cdot \mu_G(\mathrm{st}_G(v_a))}{\mu_G(C_{A_{n,r}})}\\
&=\frac{\mu_G(C_{A_{n,r}})\cdot |\pi_m(G)|^{-1}\cdot d^{-N}}{\mu_G(C_{A_{n,r}})}\\
&=|\pi_m(G)|^{-1}\cdot d^{-N}=:\epsilon(m),
\end{align*}
where we used level-transitivity of $G$ in the equality $\mu_G(\mathrm{st}_G(v_a))=d^{-N}$. As $\epsilon(m)$ only depends on $m$, and thus it only depends on $r$, this concludes the proof.
\end{proof}

\section{Geometric iterated Galois groups of rational functions}
\label{section: Geometric iterated Galois groups of rational functions}

This section is devoted to proving \cref{Theorem: mixing rational functions}. First, we give the main properties of geometric iterated Galois groups needed here. Then, we generalize the exceptional sets of Makarov and Smirnov, and show the critical role they play in applying the commutator trick to geometric iterated Galois group. Finally, we conclude the section with the proof of \cref{Theorem: mixing rational functions}.

\subsection{Geometric iterated Galois groups}

Let $K$ be a field and $f\in K(x)$ a rational function of degree $d \geq 2$. As the geometric iterated Galois group of $f$ is defined on a separable closure of $K$, we shall assume without loss of generality that~$K$ is separably closed. Consider $t$ be a transcendental element over $K$ and fix $K(t)^{\mathrm{sep}}$ a separable closure of $K(t)$. Let
$$T_t := \bigsqcup_{n \geq 0} f^{-n}(t)$$
be the tree of preimages of $t$, where $v \in f^{-n}(t)$ is adjacent to $w \in f^{-(n+1)}(t)$ if and only if $f(w) = v$. As $t$ is transcendental, the pair $(T_t, t)$ is a $d$-regular rooted tree. Moreover, the Galois action of $\Gal(K(t)^{\mathrm{sep}}/K(t))$ on each $f^{-n}(t)$ induces an action by automorphisms on~$T_t$, which defines the \textit{arboreal representation} 
$$\rho: \Gal(K(t)^{\mathrm{sep}}/K(t)) \rightarrow \Aut(T_t).$$
The image of $\rho$, denoted $G_\infty(K,f,t)$, is the \textit{geometric iterated Galois group} of $f$. The kernel of the composition $\pi_n\circ \rho$ is the subgroup
$$\mathrm{Gal}(K(t)^{\mathrm{sep}}/K_n(f,t)),$$
where $K_n(f,t) :=  K(f^{-n}(t))$. The finite extension $K_n(f,t)/K(t)$ is Galois; thus, by Galois theory, the group
$$G_n(K,f,t) := (\pi_n\circ\rho) \big(\Gal(K(t)^\mathrm{sep}/K(t))\big)$$ 
acts faithfully on the truncated tree $T_t^n$. Furthermore $K_n(f,t) \subseteq K_{n+1}(f,t)$ for all $n \geq 1$, so we may define
$$K_\infty(f,t) := \bigcup_{n\ge 1}K_n(f,t).$$
What is more
$$\ker(\rho) = \Gal(K(t)^{\mathrm{sep}}/K_\infty(f,t)),$$ and hence
$$G_\infty(K,f,t) \cong \Gal(K_\infty(f,t)/K(t)).$$
Moreover, by Galois correspondence, 
$$G_\infty(K,f,t) \cong \varprojlim_n G_n(K,f,t).$$
It can be shown that $G_\infty(K,f,t)$ is a level-transitive closed subgroup of $\mathrm{Aut}(T_t)$; see \cite[Section 1.1]{JonesArboreal}.

Now, we construct a fractal action of $G_\infty(K,f,t)$ on $T$. For that, we recall some definitions and results proved by Adams and Hyde in \cite[Section 3]{OpheliaHyde}. Given $t$ and $t'$ transcendental over $K$, a \textit{path} is a $K$-isomorphism $\lambda: K(t)^{\mathrm{sep}} \rightarrow K(t')^{\mathrm{sep}}$ such that $t^\lambda = t'$. The existence of paths is justified by Zorn's lemma. It is not hard to see that $\lambda$ induces a $K$-isomorphism between $K_\infty(f,t)$ and $K_\infty(f,t')$. 

Let $X = \set{1, \dots, d}$ and choose a bijection $\delta: X \rightarrow f^{-1}(t)$ mapping $i \mapsto t_i$. If $t$ is transcendental over $K$, then any $t_i \in f^{-1}(t)$ is also transcendental over $K$. Choose $\Lambda = \set{\lambda_i}_{i \in X}$ a set of paths such that $\lambda_i$ is a path from $t$ to $t_i$. If $g$ is an element in $G_\infty(K,f,t)$ and we have elements $t_i, t_j \in f^{-1}(t)$ such that $t_i^g = t_j$, then the restriction of $g$ to $K_\infty(f,t_i)$ is a $K(t)$-isomorphism from $K_\infty(f,t_i)$ to $K_\infty(f,t_j)$. Using the paths $\lambda_i$ and $\lambda_j$, we obtain an element $$g_i := \lambda_i \cdot g|_{K_\infty(f,t_i)} \cdot \lambda_j^{-1} \in G_\infty(K,f,t).$$
Therefore, the set of paths $\Lambda$ induces an embedding
\begin{align}
\label{equation: embedding Goo in Aut(T)}
G_\infty(K,f,t) & \hookrightarrow G_\infty(K,f,t)^d \rtimes G_1(K,f,t) \\
g & \mapsto (g_1, \dots, g_d) \pi_1(g).\nonumber
\end{align}
The set of paths $\Lambda$, together with $\delta$, induce a labeling $\phi: T_t \rightarrow T$ for $T$ the $d$-regular rooted tree obtained from the free monoid $X^\NN$. By \cref{equation: embedding Goo in Aut(T)}, the pair $(G_\infty(K,f,t), \phi)$ is self-similar. Moreover, by \cite[Theorem 7]{Ophelia}, the map 
\begin{align}
\label{equation: Goo fractal}
\st_{G_\infty(K,f,t)}(t_i) & \rightarrow G_\infty(K,f,t) \\
g & \mapsto g_i. \nonumber
\end{align}
is surjective for every $t_i\in f^{-1}(t)$ and thus the pair $(G_\infty(K,f,t), \phi)$ is fractal.

Let us write $C_f$ for the set of critical points of $f$ and $P_f:=\bigcup_{n\ge 1}f^n(C_f)$ for the post-critical set of $f$. We also write $e_f(z)$ for the local degree of $f$ at $z\in \mathbb{P}^1_K$. 

Before giving a description of the generators of the group $G_\infty(K,f,t)$, we need some definitions. Note that $K(t)$ is the field of fractions of the polynomial ring $K[t]$. Let us fix an integral separable closure of $K[t]$, i.e. let $K[t]^{\mathrm{sep}}$ be a ring generated by $K[t]$ and the roots of all monic separable polynomials with coefficients in $K[t]$.

Now, we write $\mathcal{O}_\infty:=K_\infty(f,t)\cap K[t]^{\mathrm{sep}}$ and similarly $\mathcal{O}_n:=K_n(f,t)\cap K[t]^{\mathrm{sep}}$ for every $n\ge 1$. Let us fix a prime $\mathfrak{P}$ in $\mathcal{O}_\infty$ above $(t-p)$ for some $p\in P_f$ and write $\mathfrak{P}_n:=\mathfrak{P}\cap K_n(f,t)$ for the unique prime in $\mathcal{O}_n$ below $\mathfrak{P}$. The \textit{decomposition subgroup} $D_\mathfrak{P}\le G_\infty(K,f,t)$ is precisely the stabilizer of $\mathfrak{P}$ in $G_\infty(K,f,t)$. The \textit{inertia subgroup} $I_\mathfrak{P}\le D_\mathfrak{P}$ consists of those automorphisms which restrict to the identity in the residue field $\mathcal{O}_\infty/\mathfrak{P}$. We define analogously $D_{\mathfrak{P}_n}$ and $I_{\mathfrak{P}_n}$ for any $n\ge 1$.

Based on a classical result of Grothendieck \cite{GrothendickRaynaud1971}, we obtain a description of the topological generators of $G_\infty(K,f,t)$ in terms of the post-critical points of $f$:

\begin{proposition}
Let $K$ be a separably closed field and $t$ transcendental over~$K$. Let $f\in K(x)$ be a rational function of degree $d\ge 2$ and assume that either $\mathrm{char}(K)=0$ or $\mathrm{char}(K)$ does not divide the local degree of any point in $\PP^1_K$. 
Then
\begin{enumerate}[\normalfont(i)]
\label{proposition: action inertia generators}
\item $G_\infty(K,f,t)$ is topologically generated by $\{g_p: p \in P_f\}$, where each $g_p$ corresponds to a topological generator of an inertia subgroup $I_\mathfrak{P}$ for some prime~$\mathfrak{P}$ in $\mathcal{O}_\infty$ above $(t-p)$. Moreover, there exists an ordering of the topological generators such that
$$\prod_{p \in P_f} g_p  = 1.$$

\item For each $t_i \in f^{-1}(t)$, we have $\mathfrak{P} \cap K(t_i) = (t_i-q)$ for some $q \in f^{-1}(p)$, and the $g_p$-orbit of $t_i$ is a cycle of length $e_f(q)$.

\item Let $t_i$ and $q$ as in $\mathrm{(ii)}$ and let $\mathcal{Q}$ be a prime in $\mathcal{O}_\infty$ above $(t-q)$. Then, there exists an element $g$ in $G_\infty(K,f,t)$ conjugated to $g_p$ in $G_\infty(K,f,t)$ such that $g^{e_f(q)}$ fixes $t_i$ and $(g^{e_f(q)})_i$ generates topologically the inertia subgroup~$I_\mathcal{Q}$.
\end{enumerate}
\end{proposition}

\begin{proof}
Let us prove (i) first. We write $P_n:=\bigcup_{i = 1}^n f^i(C_f)$. Let $K(t)_{P_n}$ be the maximal tamely ramified extension of $K(t)$ in $K(t)^{\mathrm{sep}}$ which is ramified only above $P_n$. Note that every inertia subgroup $I_\mathfrak{P}$ is procyclic for any prime $\mathfrak{P}$ in $K(t)_{P_n}$ above $(t-p)$, as the extension is tamely ramified. It is a classical result of Grothendieck (see \cite[Corollary 3.9 Exposé X and Corollary 5.2 Exposé XII]{GrothendickRaynaud1971} or \cite[Theorem~4.9.1]{Szamuely2009}) that $$\mathrm{Gal}(K(t)_{P_n}/K(t)) = \overline{\langle \tau_p : \prod_{p \in P_n} \tau_p = 1 \rangle},$$
where each $\tau_p$ is a topological generator of an inertia subgroup $I_\mathfrak{P}$ for some prime~$\mathfrak{P}$ in $K(t)_{P_n}$ above $(t-p)$.

By the chain rule
$$e_{f^n}(z) = \prod_{i = 0}^{n-1} e_f(f^i(z))$$
for any $z\in \mathbb{P}^1_K$. Therefore, by hypothesis $\mathrm{char}(K) \nmid e_{f^n}(p)$ for $p\in P_n$, so for every $n\ge 1$ the extension $K_n(f,t)/K(t)$ is tamely ramified and it only ramifies over $P_n$. As $K_n(f,t)/K(t)$ is Galois, we obtain
\begin{align*}
G_n(K,f,t) = \mathrm{Gal}(K_n(f,t)/K(t))\cong \mathrm{Gal}(K(t)_{P_n}/K(t))/\mathrm{Gal}(K(t)_{P_n}/K_n(f,t))
\end{align*}
and consequently $G_n(K,f,t)$ is generated by the set $\{\tau'_p: p \in P_n\}$, where $\tau'_p$ is the projection of the element $\tau_p$. As this applies to every $n\ge 1$ and we have $$G_\infty(K,f,t)=\varprojlim_{n}G_n(K,f,t),$$
we may find a generating set $\{g_p:p\in P_f\}$ of $G_\infty(K,f,t)$ such that $g_p$ projects to~$\tau'_p$ for every $p \in P_f$. Furthermore, the element $g_p$ generates the inertia subgroup of~$I_\mathfrak{P}$ for some prime~$\mathfrak{P}$ in $K_\infty(f,t)$ above $(t-p)$ and 
$$\prod_{p \in P_f} g_p = 1$$ 
as wanted.

To prove (ii) we proceed as follows. Let $k_1:=\mathcal{O}_1/\mathfrak{P}_1$. Then
$$\prod_{t_i\in f^{-1}(t)}(x-t_i)=f(x)-t\equiv f(x)-p=\prod_{q\in f^{-1}(p)}(x-q)^{e_f(q)},$$
where the congruence is given by considering coefficients modulo $\mathfrak{P}_1$. Let $t_i\in f^{-1}(t)$. Then, by the unique factorization property in $k_1[x]$ we obtain that
\begin{align*}
    t_i\equiv q &\text{ mod }\mathfrak{P}_1
\end{align*}
for some $q\in f^{-1}(p)$. Therefore, the $I_\mathfrak{P}$-orbit of $t_i$ is contained in the set of preimages $t_j\in f^{-1}(t)$ such that
$$t_j\equiv q \text{ mod }\mathfrak{P}_1.$$
As there are precisely $e_f(q)$ such preimages $t_j$, it only remains to show that the $I_\mathfrak{P}$-orbit of $t_i$ is of size at least $e_f(q)$.  By the orbit-stabilizer theorem, this is equivalent to $$|I_{\mathfrak{P}}:\mathrm{Stab}_{I_{\mathfrak{P}}}(t_i)|=|I_{\mathfrak{P}_1}:\mathrm{Stab}_{I_{\mathfrak{P}_1}}(t_i)|\ge e_f(q).$$
Now, we have the congruence
$$(t_i-q)\equiv 0 \text{ mod }\mathfrak{P}_1,$$
so $(t_i-q)$ is below $\mathfrak{P}_1$. Applying the Dedekind-Kummer theorem to the extension $K(t_i)/K(t)$ yields 
$$(t_i-q)=\mathfrak{P}_1\cap K(t_i)\quad \text{and}\quad e((t_i-q)|(t-p))=e_f(q).$$
As $\mathrm{Stab}_{I_{\mathfrak{P}_1}}(t_i)=I_{\mathfrak{P}_1}\cap \mathrm{Gal}(K_1(f,t)/K(t_i))$ we get
\begin{align*}
    |I_{\mathfrak{P}_1}:\mathrm{Stab}_{I_{\mathfrak{P}_1}}(t_i)|&\ge e((t_i-q)|(t-p))=e_f(q).
\end{align*}

Let us prove (iii) now. By (ii), the element $g_p^{e_f(q)}$ fixes $t_i$. As $(t-p)$ has ramification index $e_f(q)$ over $(t_i - q)$ and $K_\infty(f,t_i) \subseteq K_\infty(f,t)$, the restriction $g_p^{e_f(q)}|_{K_\infty(f,t_i)}$ is a topological generator of the inertia subgroup of $\mathfrak{P} \cap K_\infty(f,t_i)$ which is above $(t_i - q)$. Conjugating by the path $\lambda_i$, we obtain that $(g_p^{e_f(q)})_i$ is a topological generator of the inertia subgroup of $\lambda_i^{-1}(\mathfrak{P} \cap K_\infty(f,t_i))$ which is above $(t - q)$. As the action of $G_\infty(K,f,t)$ on the prime ideals above $(t-q)$ is transitive, there exists $h \in G_\infty(K,f,t)$ such that $h$ maps $\lambda_i^{-1}(\mathfrak{P} \cap K_\infty(f,t_i))$ to~$\mathcal{Q}$ and consequently the element $h^{-1} (g_p^{e_f(q)})_i  h$ is a topological generator of the inertia subgroup $I_\mathcal{Q}$. Finally, using that the map $g\mapsto g_i$ is surjective, there exists $\widetilde{h} \in \mathrm{st}_{G_\infty(K,f,t)}(t_i)$ such that $(\widetilde{h})_i = h$. Set $g := \widetilde{h}^{-1} g_p \widetilde{h}$. Then $g^{e_f(q)} = \widetilde{h}^{-1} g_p^{e_f(q)} \widetilde{h}$ fixes~$t_i$ and $(g^{e_f(q)})_i$ is a topological generator of $I_\mathcal{Q}$.
\end{proof}

\begin{remark}
\label{remark: sections of generators}
In what follows, we shall make a slight abuse of notation and consider $G_\infty(K,f,t)$ as a subgroup of $\mathrm{Aut}(T)$ via the fractal action $(G_\infty(K,f,t),\phi)$. Then,
\cref{proposition: action inertia generators} yields a recursive method to compute the generators of the geometric iterated Galois group $G_\infty(K,f,t)$, a property that will be used in \cref{theorem: good generators are seen}. If $p \in P_f$ and we consider the prime $\mathfrak{P} \mid (t-p)$ such that $g_p$ corresponds to the topological generator of the inertia subgroup $I_\mathfrak{P}$. Then, as $G_\infty(K,f^n,t)=G_\infty(K,f,t)$,  applying \cref{proposition: action inertia generators}\textcolor{teal}{(ii)} to $f^n\in K(x)$, for an element $s \in f^{-n}(t)$, we have $\mathfrak{P} \cap K(s) = (s-q)$ for some $q \in f^{-n}(p)$. We write $v$ for the vertex in $T$ corresponding to the preimage $s\in f^{-n}(t)$ via the labeling $\phi$. We consider two cases:
\begin{enumerate}[\normalfont(i)]
\item If $q \notin P_f$, then the inertia subgroup of $I_\mathcal{Q}$ is trivial for any $\mathcal{Q}$ above $(t-q)$. Therefore  $$g_p^{e_{f^n}(q)}\in \mathrm{st}_{G_\infty(K,f,t)}(v)\le \mathrm{Aut}(T)$$ and $(g_p^{e_{f^n}(q)})|_v = 1$.
\item If $q \in P_f$, let $\mathcal{Q}$ be the prime ideal above $(t-q)$ whose inertia subgroup is topologically generated by $g_{q}$. By \cref{proposition: action inertia generators}\textcolor{teal}{(iii)}, there exists $g\in G_\infty(K,f,t)$ conjugate to $g_p$ in $G_\infty(K,f,t)$ fixing $v$ and such that
$$\overline{\langle (g^{e_{f^n}(q)})|_v\rangle} = \overline{\langle g_{q}\rangle}.$$
\end{enumerate}
\end{remark}

We shall see that many of the generators of the geometric iterated Galois group of a rational function satisfy the assumptions in \cref{lemma: commutator trick}. For that, we generalize the notion of exceptional sets defined by Makarov and Smirnov in \cite{MakarovSmirnov}. This is motivated by \cref{lemma: commutator trick}, as the main obstruction to apply the commutator trick in \cref{lemma: commutator trick} lies in this generalization of exceptional sets.

\subsection{Exceptional sets and generalizations}

Let us first recall the Riemann-Hurwitz theorem:

\begin{theorem}[Riemann-Hurwitz]
\label{theorem: Riemann-Hurwitz}
Let $K$ be a separably closed field and $f\in K(x)$ a rational function of degree $d \geq 2$. Then
$$2(d-1) = \sum_{z \in \mathbb{P}^1_K} (e_f(z) -1).$$ 
Moreover, if $f$ is a polynomial, then 
$$d-1 = \sum_{z \in K} (e_f(z) -1).$$ 
\end{theorem}

In \cite{MakarovSmirnov}, Makarov and Smirnov introduced the notion of an exceptional set for a complex rational function $f$. A set $\Sigma\subseteq P_f\subset \mathbb{P}_K^1$ is said to be \textit{exceptional} if 
$$\Sigma=f^{-1}(\Sigma)\setminus C_f.$$

Note that the union of exceptional sets is exceptional; thus the union of all the exceptional sets is exceptional and it is uniquely determined by $f$. In the following, the set $\Sigma_f$ will denote the unique maximal exceptional set for $f$.

We consider the following generalization of the exceptional set, where we allow the exceptional points to be both critical and post-critical. We say that a set $\Upsilon \subseteq P_f \subset \mathbb{P}_K^1$ is \textit{critically exceptional} if 
$$\Upsilon = f^{-1}(\Upsilon)\setminus ((C_f \cup P_f) \setminus \Upsilon).$$

As for exceptional sets, the union of critically exceptional sets is critically exceptional. Hence, we write $\Upsilon_f$ for the unique maximal critically exceptional set for~$f$. 

Let us define also
$$\Delta_f := \{p \in P_f: f^{-1}(p) \subseteq C_f \cup P_f\}\subset \mathbb{P}_K^1,$$
i.e. the set of post-critical points whose preimages are all critical or post-critical. Clearly 
$$\Sigma_f \subseteq \Upsilon_f \subseteq \Delta_f.$$

A simple application of the Riemann-Hurwitz formula shows that $\Upsilon_f$ and $\Delta_f$ satisfy the same size restrictions as $\Sigma_f$:

\begin{lemma}
\label{lemma: almost exceptional set}
Let $f\in K(x)$ be a rational function of degree $d\ge 2$. Then $$\#\Delta_f\le 4,$$
and if $f$ is a polynomial, we further get
$$\#(\Delta_f\cap K)\le 2.$$
Moreover $\#\Upsilon_f =  4$ if and only if
\begin{enumerate}[\normalfont(i)]
\item every critical point is of degree 2;
\item $\Sigma_f = \Upsilon_f = \Delta_f = P_f$ and
\item $f^{-1}(\Upsilon_f) = C_f \sqcup P_f$.
\end{enumerate}
Similarly, if $f$ is a polynomial  $\#(\Upsilon_f\cap K) =  2$ if and only if 
\begin{enumerate}[\normalfont(i)]
\item every critical point in $K$ is of degree 2;
\item $\Sigma_f = \Upsilon_f \cap K = \Delta_f \cap K = P_f \cap K$ and
\item $f^{-1}(\Upsilon_f \cap K) = (C_f \cap K) \sqcup (P_f \cap K)$. 
\end{enumerate}
\label{lemma: bound for Deltaf and Upsilonf}
\end{lemma}

\begin{proof}
We prove the statement for rational functions, as the polynomial case follows from the same arguments applying the statement for polynomias in \cref{theorem: Riemann-Hurwitz}. Let us define a map
\begin{align*}
\varphi: f^{-1}(\Delta_f) & \to \Delta_f \cup C_f \\
p & \mapsto \varphi(p),
\end{align*}
where $\varphi(p)$ is an element both in the backward orbit of $p$ and in $\Delta_f \cup C_f$, which is at minimal distance from $p$. In other words, if $f^n(\varphi(p)) = p$, then $\varphi(p)$ satisfies that for any other element $q \in \Delta_f \cup C_f$ such that $f^m(q) = p$, then $m \geq n$. Such an element $\varphi(p)$ always exists as $f^{-1}(\Delta_f) \subseteq P_f \cup C_f$ and every post-critical point has a critical point in its backward orbit. As there might be different choices for the map $\varphi$, let us fix one for the remainder of the proof.

We claim that the map $\varphi$ is injective. Indeed, suppose that $\varphi(p_1) = \varphi(p_2) = q$ for two elements $p_1,p_2\in f^{-1}(\Delta_f)$. Then, there exist $n_1, n_2 \geq 0$ such that $f^{n_i}(q) = p_i$ for $i = 1,2$ and $n_1, n_2$ are minimal by the definition of $\varphi$. If $n_1= n_2$, clearly $p_1 = p_2$ and $\varphi$ is injective. Thus, let us assume now by contradiction, without loss of generality, that $n_1 < n_2$. As $p_1 \in f^{-1}(\Delta_f)$, then $f(p_1) = f^{n_1+1}(q) \in \Delta_f$. Hence $p_2 = f^{n_2 - (n_1+1)}(f(p_1))$, which contradicts the minimality of $n_2$. Therefore $\varphi$ is injective, which yields
\begin{align}
\# f^{-1}(\Delta_f) & \leq \#(\Delta_f \cup C_f)\leq \# \Delta_f + \# C_f \leq \#\Delta_f + 2(d-1).
\label{equation: upper bound preimage Deltaf}
\end{align}
On the other hand
\begin{align}
\label{equation: lower bound preimage Deltaf}
\# f^{-1}(\Delta_f) &= \sum_{p \in \Delta_f} \# f^{-1}(p) = \sum_{p \in \Delta_f} \left( d - \sum_{c \in f^{-1}(p) \cap C_f} (e_f(c) - 1) \right) \\
&=d \# \Delta_f - \sum_{c \in f^{-1}(\Delta_f) \cap C_f} (e_f(c) - 1) \notag \\
&\ge d \# \Delta_f - \sum_{c \in C_f} (e_f(c) - 1) \notag \\
&=  d \# \Delta_f - 2(d-1),\notag
\end{align}
where the last equality follows from Riemann-Hurwitz. 

Combining \textcolor{teal}{Inequalities (}\ref{equation: upper bound preimage Deltaf}\textcolor{teal}{)} and \textcolor{teal}{(}\ref{equation: lower bound preimage Deltaf}\textcolor{teal}{)}, we obtain
$$d \# \Delta_f - 2(d-1)\le \#\Delta_f+2(d-1),$$
and thus
$$\# \Delta_f \leq 4.$$

To prove the if and only if statement, note first that the equality $\# \Delta_f = 4$ holds if and only if all the inequalities in \textcolor{teal}{(}\ref{equation: upper bound preimage Deltaf}\textcolor{teal}{)} and \textcolor{teal}{(}\ref{equation: lower bound preimage Deltaf}\textcolor{teal}{)} are equalities, i.e. if and only if
\begin{enumerate}[(a)]
\item every critical point is of degree 2;
\item $\Delta_f$ and $C_f$ are disjoint;
\item $C_f \subseteq f^{-1}(\Delta_f)$ and
\item $\# f^{-1}(\Delta_f) = \# (\Delta_f \cup C_f)$. 
\end{enumerate}

Suppose first that $\# \Upsilon_f = 4$. Then (a) is precisely (i). To show (ii) and (iii) note first that $\Upsilon_f = \Delta_f$. Note that by definition of $\Upsilon_f$, we have $\Upsilon_f \subset f^{-1}(\Upsilon_f)$ and thus $f^n(\Upsilon_f) \subseteq \Upsilon_f$ for every $n \geq 1$. Combining this with (c), we obtain that $P_f \subseteq \Upsilon_f$ and consequently $\Upsilon_f = P_f$. By (b), and the fact that $\Upsilon_f = \Delta_f =  P_f$ we conclude that $C_f$ and $P_f$ are disjoint. Therefore, by the definition of $\Upsilon_f$ we get $\Upsilon_f = f^{-1}(\Upsilon_f) \setminus C_f$, and hence
$$f^{-1}(\Upsilon_f) = \Upsilon_f \sqcup C_f.$$
Thus $\Upsilon_f$ is exceptional and $\Sigma_f = \Upsilon_f$, establishing both (ii) and (iii).

Conversely, if $\Upsilon_f$ satisfies conditions (i)-(iii), then $\Upsilon_f=\Delta_f$ and (a)-(d) are also satisfied. Therefore $\#\Upsilon=\#\Delta_f=4$ concluding the proof.
\end{proof}

\subsection{Commutator trick}

Now, we prove the key observation, namely that generators corresponding to post-critical points not in~$\Upsilon_f$ may be seen from any level-stabilizer by projecting several levels further down in the tree. The following remark is used in the proof. We record it here for the convenience of the reader:
\begin{remark}
\label{remark: conjugates fix 1}
Let $g\in \mathrm{Aut}(T)$ be such that $$g^r\in \mathrm{st}(v)\quad\text{and}\quad g^r|_v=1$$
for some $r\ge 1$ and some $v\in \mathcal{L}_n$ with $n\ge 1$. Then for any $h\in \mathrm{Aut}(T)$, if we set $k:=h^{-1}gh$, we get
$$k^r\in \mathrm{st}(v^h)\quad\text{and}\quad k^r|_{v^h}=h^{-1}|_{v^h}\cdot h|_v=(h|_v)^{-1}\cdot h|_v=1.$$
In other words, any element $k$ conjugate to $g$ in $\mathrm{Aut}(T)$ satisfies that $k^r$ has trivial section at a vertex at level $n$ fixed by $k^r$.
\end{remark}

\begin{theorem}
\label{theorem: good generators are seen}
Let $K$ be a separably closed field and $f\in K(x)$ a rational function of degree $d \ge 2$. Let $G:=G_\infty(K,f,t)$ and assume that $\mathrm{St}_G(n)$ acts level-transitively on the subtrees rooted at level $n$ for any $n \geq 1$. Then, for any $n \geq 1$ and any $v \in T$ such that $|v| \geq n+4$, we have
$$\mathrm{St}_G(n)_v\ge \overline{\langle g_p : p \in P_f \setminus \Upsilon_f\rangle}.$$
\end{theorem}

\begin{proof}
We may assume that $P_f\setminus \Upsilon_f\ne \emptyset$, as otherwise there is nothing to prove. Let $p\in P_f\setminus \Upsilon_f$. Assume first that there exists $n \geq 1$ such that $f^n(p) \notin \Delta_f$. We write $q = f^n(p)$ for such a minimal $n$. By \cref{remark: sections of generators}, there exist $g\in G$ and a vertex $v \in \mathcal{L}_n$ such that $g$ is conjugate to~$g_q$ in $G$ and 
$$g^{e_{f^n}(p)}\in \mathrm{st}_G(v),\quad g^{e_{f^n}(p)}|_v = h_p\quad\text{and}\quad \overline{\langle h_p\rangle}=\overline{\langle g_p\rangle}\le G_.$$
As $q \notin \Delta_f$, it has a preimage outside $C_f \cup P_f$. Hence, by \cref{remark: sections of generators}, there exists $x \in \mathcal{L}_1$ fixed by $g_q$ such that $g_q|_x = 1$. Furthermore, by \cref{remark: conjugates fix 1}, there is another vertex $y\in \mathcal{L}_1$ fixed by $g$ and where $g|_y=1$. As $n$ was minimal, we have $\{f^i(p)\}_{i = 1}^{n-1} \subseteq \Delta_f$. Thus $n\le 4$ by \cref{lemma: almost exceptional set}. Then $g^{e_{f^n}(p)}$ and $h_p$ satisfy the assumptions in \cref{lemma: commutator trick} with $N=4$. 

Let us assume now that $f^{n}(p)\in \Delta_f$ for all $n\ge 1$. Then $I:=\{f^n(p)\}_{n\ge 1}$ is a critically exceptional set, as
$$f^{-1}(I)=I\setminus ((C_f\cup P_f)\setminus I).$$
In other words $I\subseteq \Upsilon_f$. Furthermore $p\notin \Delta_f$, as otherwise $p\in \Upsilon_f$ by the above argument.

Now, by \cref{lemma: almost exceptional set} we have $\#I\le \#\Upsilon_f\le 3$ as $P_f\setminus \Upsilon_f\ne \emptyset$. Thus, the point $p$ is strictly preperiodic. Let $1\le \ell\le 3$ be minimal such that $q:=f^{\ell}(p)\in \mathcal{C}$ for the unique cycle $\mathcal{C}$ in $I$. Let us also write $r:=\#\mathcal{C}\ge 1$. Clearly $\ell+r\le 3+1=4$. As $q=f^{\ell+r}(p)$, by \cref{remark: sections of generators}, there exist $g\in G$ conjugate to $g_q$ in $G$ and $v\in \mathcal{L}_{r+\ell}$ such that
$$g^{e_{f^{\ell+r}}(p)}\in \mathrm{st}_G(v),\quad g^{e_{f^{\ell+r}}(p)}|_v=h_p\quad \text{and}\quad \overline{\langle h_p\rangle}=\overline{\langle g_p\rangle}\le G.$$
Now, as $p\notin \Delta_f$, there is $z\in f^{-1}(p)\setminus (C_f\cup P_f)$. Thus $q=f^{1+\ell}(z)$, and by \cref{remark: sections of generators}, there is a vertex $u\in \mathcal{L}_{1+\ell}$ such that
$$g_q^{e_{f^{1+\ell}}(z)}\in \mathrm{st}_G(u)\quad \text{and}\quad g_q^{e_{f^{1+\ell}}(z)}|_u=1.$$
Hence, by \cref{remark: conjugates fix 1}, there is a vertex $w\in \mathcal{L}_{1+\ell}$ such that
$$g^{e_{f^{1+\ell}}(z)}\in \mathrm{st}_G(w)\quad \text{and}\quad g^{e_{f^{1+\ell}}(z)}|_w=1.$$
Note that $e_{f^{1+\ell}}(z)=e_{f^{\ell}}(p)$, and by the chain rule
$$e_{f^{\ell+r}}(p)=e_{f^{\ell}}(p)\cdot e_{f^{r}}(f^{\ell}(p))=e_{f^{1+\ell}}(z)\cdot e_{f^{r}}(f^{\ell}(p)),$$
so $e_{f^{1+\ell}}(z)$ divides $e_{f^{\ell+r}}(p)$. Therefore, as $1+\ell\le r+\ell\le 4$, the elements $g^{e_{f^{\ell+r}}(q)}$ and $h_p$ satisfy the assumptions in \cref{lemma: commutator trick} with $N=4$ in this case too. Thus \cref{lemma: commutator trick} yields $$\mathrm{St}_G(n)_v \ge \overline{\langle h_p : p \in P_f \setminus \Upsilon_f\rangle}= \overline{\langle g_p : p \in P_f \setminus \Upsilon_f\rangle}$$
for every $v\in T$ such that $|v|\ge n+4$.
\end{proof}

\begin{figure}[H]
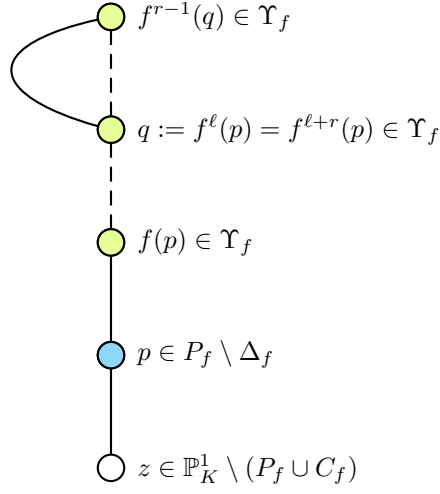

        \centering
  \def\dedge{\ncline[linestyle=dashed]}
      \psTree{\TC[radius=5pt,name=x1]}
            \psTree{\TC[radius=5pt,name=x2, edge=\dedge]}
        \psTree{\TC[radius=5pt,name=x3,edge=\dedge]}
            \psTree{\TC[radius=5pt,name=x4]}
            \TC[radius=5pt,name=x5]
         \endpsTree 
            \endpsTree 
        \endpsTree 
      \endpsTree
\uput{10pt}[0](x1){$f^{r-1}(q)\in \Upsilon_f$}%
\uput{10pt}[0](x2){$q:=f^{\ell}(p)=f^{\ell+r}(p)\in \Upsilon_f$}%
\uput{10pt}[0](x3){$f(p)\in \Upsilon_f$}%
\uput{10pt}[0](x4){$p\in P_f\setminus \Delta_f$}%
\uput{10pt}[-4](x5){$z\in \mathbb{P}_K^1\setminus(P_f\cup C_f)$}%
\pscurve[linecolor=black]{-}(x1)(-1.5,-0.7)(x2)%
\Ball[lime](x1)%
\Ball[lime](x2)%
\Ball[lime](x3)%
\Ball[cyan](x4)%
        \caption{A graphical representation of the $f$-orbit of $p\in P_f\setminus \Delta_f$ in the proof of \cref{theorem: good generators are seen}, where $z\in f^{-1}(p)$.}
        \label{figure: commutator trick}
\end{figure}

\begin{remark}
The delay constant $N=4$ may not be sharp, but it is convenient to state. A better bound for the delay constant can be computed explicitly for each~$f$ following the argument in the proof of \cref{theorem: good generators are seen}. However, unless the action on the first level is abelian, we will not be able to show $N=0$ (i.e. super strong fractality) using the commutator trick.
\end{remark}

\begin{remark}
\cref{theorem: good generators are seen} does not imply that $g_p\notin \mathrm{St}_G(n)_v$ for $p\in \Upsilon_f$. In fact, let us assume that there exists $\widetilde{p}\in P_f\setminus \Delta_f$ and $n\ge 1$ such that $f^n(\widetilde{p})=p$ and consider $q$ as in the proof of \cref{theorem: good generators are seen}.  If $e_{f^n}(\widetilde{p})\mid e_{f^{r}}(q)^m$ for some $m\ge 1$, then a similar argument to the one in the proof of \cref{theorem: good generators are seen} shows that 
$$\mathrm{St}_G(n)_v\ge \overline{\langle g_p\rangle}.$$
Observe that the delay constant may still be taken to be $N=4$, as $\widetilde{p}$ is a post-critical point in the backward orbit of $p$ not in $\Upsilon_f$ and at minimal distance to~$p$.
\end{remark}

\subsection{Mixing property}

As a corollary of \cref{theorem: good generators are seen}, we obtain \cref{Theorem: mixing rational functions}, i.e. the mixing property for a rational function $f$ satisfying the martingale assumption and $\#\Upsilon_f\le 1$:

\begin{proof}[Proof of \cref{Theorem: mixing rational functions}]
Let $n\ge 1$ and consider a vertex $v \in T$ such that $|v| \geq n + 4$. We write $G:=G_\infty(K,f,t)$. By \cref{proposition: martingale characterization}, the level-stabilizers $\mathrm{St}_G(n)$ act level-transitively on the subtrees rooted at level $n$ as the fixed-point process of $G$ is a martingale. Then, as $\# \Upsilon_f\le 1$, all the generators but possibly $g_p$ for $p\in \Upsilon_f$ are in $\mathrm{St}_G(n)_v$ by \cref{theorem: good generators are seen}. However, as
$$\prod_{q\in P_f}g_q=1$$
and $\mathrm{St}_G(n)_v$ is closed, we further obtain $g_p\in \mathrm{St}_G(n)_v$. This yields 
$$\mathrm{St}_G(n)_v=G.$$
In particular $G$ is mixing with delay constant $N=4$.
\end{proof}

Note that \textcolor{teal}{Theorems} \ref{Theorem: main result FPP intro} and \ref{Theorem: mixing rational functions} further extend \textcolor{teal}{Corollaries} \ref{Corollary: periodic points 1} and \ref{Corollary: periodic points 2} for rational functions satisfying the assumptions of \cref{Theorem: mixing rational functions}.

\section{The polynomial case}
\label{section: topological polynomials}

We conclude the paper with the proof of \cref{Theorem: main IMG intro}. We first show that monodromy at infinity yields the martingale condition and then we show that the exceptional case corresponds to the euclidean orbifolds of type $(2,2,\infty)$. 

\subsection{Mixing property}

Let us assume now that $f\in K[x]$ is a polynomial of degree $d \geq 2$, where $K$ is a separably closed field. As $f$ is a polynomial, then $\infty$ is a totally ramified fixed point, i.e. $e_f(\infty) = d$ and $f(\infty) = \infty$. In particular $\infty \in P_f$. Let us write $g_\infty$ for the topological generator which corresponds to a topological generator of the inertia subgroup over $\infty$. We have the following:

\begin{lemma}
\label{lemma: IGG polynomials martingale}
Let $K$ be a separably closed field and $f\in K[x]$ a polynomial of degree $d \ge 2$. If we write $G := G_\infty(K,f,t)$, we have:
\begin{enumerate}[\normalfont(i)]
\item for all $n \geq 1$, the power $g_\infty^{d^n} \in \mathrm{St}_G(n)$ and $g_\infty^{d^n}|_v^1$ acts like a $d$-cycle on the immediate descendants of $v$ for every $v\in \mathcal{L}_n$;
\item the fixed-point process of $G$ is a martingale.
\end{enumerate}
\end{lemma}
\begin{proof}
Applying \cref{proposition: action inertia generators} to $f^n$, we get $\pi_n(g_\infty)$ acts like a $d^n$-cycle on $\mathcal{L}_n$ as the point $\infty$ is totally ramified for the polynomial $f^n$. Moreover, for any $v\in \mathcal{L}_n$, the section $g_\infty^{d^n}|_v$ is conjugate to a generator of the subgroup $\langle g_\infty\rangle$ by \cref{proposition: action inertia generators} as $g_\infty^{d^n}$ fixes $v$. In particular, the label $g_\infty^{d^n}|_v^1$ must be a $d$-cycle, proving (i). Now, (ii) follows from \cref{proposition: martingale characterization} as $g_\infty$ is a level-transitive element by (i).
\end{proof}

As a result of \cref{lemma: IGG polynomials martingale}, we obtain an improvement of \cref{Theorem: mixing rational functions} for polynomials:

\begin{theorem}
\label{lemma: mixing polynomial}
Let $K$ be a separably closed field and a polynomial $f\in K[x]$ of degree $d \geq 2$. Set $t$ transcendental over $K$. If $\# (\Upsilon_f \cap K) \leq 1$, then $G_\infty(K,f,t)$ is mixing.
\end{theorem}
\begin{proof}
Let $n \geq 1$ and $v \in T$ any vertex such that $|v| \ge n + 4$. By \cref{lemma: IGG polynomials martingale}, we have 
$$\mathrm{St}_G(n)_v\ge \overline{\langle g_\infty\rangle}.$$
Thus, as $\# (\Upsilon_f \cap K) \leq 1$, there is at most one generator $g_p$ not seen in $\mathrm{St}_G(n)_v$ by \cref{theorem: good generators are seen}. Then \cref{Theorem: mixing rational functions} yields that $G_\infty(K,f,t)$ is mixing.
\end{proof}

\begin{remark}
    In fact, taking advantage of the extra restriction on the size of $\Upsilon_f\cap K$ and $\Delta_f\cap K$ in \cref{lemma: bound for Deltaf and Upsilonf}, it is easy to see that the delay constant may be improved from $N=4$ to $N= 2$ in \cref{lemma: mixing polynomial}.
\end{remark}

\subsection{Euclidean orbifolds of type $(2,2,\infty)$}

It is a well-known fact that exceptional complex polynomials with $\#\Sigma_f=2$ are linearly conjugate (up to sign) to Chebyshev polynomials; see \cite[Proposition 8.4]{JonesAMS} for an algebraic proof and \cite[Theorem~19.9]{MilnorBook} for a geometric proof. Here, for an arbitrary separably closed field~$K$, we obtain a characterization of polynomials with $\#\Sigma_f=2$ in terms of Thurston orbifolds. 

Given a post-critically finite rational function $f\in K(x)$ and a point $z\in \mathbb{P}_K^1$, we say that a periodic orbit $\{f^i(z)\}_{i=0}^{n-1}$ is \textit{super-attracting} if it contains a critical point of $f$. We define the \textit{Thurston orbifold} of $f$ as $(\mathbb{P}_K^1,\nu_f)$ for the map $\nu_f:\mathbb{P}_K^1\to \mathbb{N}\cup \{\infty\}$, where
$$\nu_f(z):=\mathrm{lcm}\{\nu_f(w)e_f(w) : w\in f^{-n}(z) \text{ for some }n\ge 1\},$$
unless $z$ belongs to a super-attracting periodic orbit, in which case $\nu_f(z)=\infty$. The finite tuple $(\nu_f(z))_{z\in P_f}$ is the \textit{type} of the orbifold. Furthermore, the \textit{Euler characteristic} $\chi$ of the orbifold $(\mathbb{P}_K^1,\nu_f)$ is
$$\chi(\nu_f):=2-\sum_{z\in P_f}\left(1-\frac{1}{\nu_f(z)}\right).$$
A rational function $f\in K(x)$ is said to be \textit{euclidean} if $\chi(\nu_f)=0$ and \textit{hyperbolic} if $\chi(\nu_f)<0$. 

We show now a strong relation between exceptional sets and orbifolds:

\begin{lemma}
\label{lemma: euclidean and exceptional}
Let $K$ be a separably closed field and $f\in K[x]$ be a polynomial. Then, the following are equivalent:
\begin{enumerate}[\normalfont(i)]
\item $\#\Sigma_f=2$;
\item $f$ has a euclidean orbifold of type $(2,2,\infty)$.
\end{enumerate}
\end{lemma}
\begin{proof}
Let us assume first that we have $\#\Sigma_f=2$. By \cref{lemma: almost exceptional set}, the polynomial~$f$ has exactly 2 post-critical points $p_1,p_2\in K$ which are not critical, and every critical point in $K$ is of degree 2. Furthermore, as $f$ is a polynomial the point at infinity is a super-attracting fixed point. Thus 
$$(\nu_f(p_1),\nu_f(p_2),\nu_f(\infty))=(2,2,\infty).$$

Let us assume now that $f$ has a euclidean orbifold of type $(2,2,\infty)$. Then $f$ has exactly 3 post-critical points $p_1,p_2,\infty$. Furthermore, as $f$ is a polynomial, the point~$\infty$ is a super-attracting fixed point. Then, every other critical point different to $\infty$ is mapped to $\{p_1,p_2\}$. By the orbifold type, every critical point is of degree 2 and $p_1$ and $p_2$ are not critical themselves. Then
\begin{align*}
\Sigma_f&=\{p_1,p_2\}.\qedhere
\end{align*}
\end{proof}

Geometric iterated Galois groups of polynomials with euclidean orbifolds of type $(2,2,\infty)$ have a dense dihedral subgroup:

\begin{lemma}
\label{lemma: euclidean implies dihedral}
Let $K$ be a separably closed field and $f\in K[x]$ be a polynomial of degree $d\ge 2$ with euclidean orbifold of type $(2,2,\infty)$. Then $G_\infty(K,f,t)$ is the closure in $\mathrm{Aut}(T)$ of the infinite dihedral group $\langle g_{p_1},g_{p_2}\rangle$.
\end{lemma}
\begin{proof}
Let $P_f=\{p_1,p_2,\infty\}$, where neither $p_1$ nor $p_2$ are critical by \cref{lemma: euclidean and exceptional}. By \cref{proposition: action inertia generators}, the group $G_\infty(K,f,t)$ is topologically generated by $\{g_{p_1},g_{p_2}\}$ as 
$$g_\infty^{-1}=g_{p_1}g_{p_2}.$$
Furthermore, both $g_{p_1}$ and $g_{p_2}$ are of order 2 by \cref{remark: sections of generators}, and its product~$g_{p_1}g_{p_2}$ is level-transitive (and thus of infinite order) by \cref{lemma: IGG polynomials martingale}\textcolor{teal}{(i)}. Hence
\begin{align*}
\langle g_{p_1},g_{p_2}\rangle\cong D_\infty&:=\langle x,y\mid x^2=y^2=1\rangle.\qedhere
\end{align*}
\end{proof}

Polynomials over $K$ with a euclidean orbifold of type $(2,2,\infty)$ belong to a family of polymomials generalizing Chebyshev polynomials. A polynomial $f\in K[x]$ is a \textit{standard twisted Chebyshev polynomial} of degree $d\ge 2$ if it satisfies
$$f\left(x+\frac{1}{x}\right)=x^d+\frac{\zeta}{x^d}$$
for some $\zeta$ satisfying $\zeta^{d-1}=1$. A polynomial $f$ linearly conjugate to a standard twisted Chebyshev polynomial is called a \textit{twisted Chebyshev polynomial}. Based on the results of Adams in \cite{Ophelia}, we prove the following

\begin{proposition}
\label{proposition: twisted chebyshev}
Let $K$ be a separably closed field and $f\in K[x]$ be a polynomial of degree $d\ge 2$ with euclidean orbifold of type $(2,2,\infty)$. Then $f$ is a twisted Chebyshev polynomial. 
\end{proposition}
\begin{proof}
By \cref{lemma: euclidean implies dihedral}, the generator at infinity $g_\infty$ generates topologically a pro-cyclic open subgroup $H$ of index 2 in $G_\infty(K,f,t)$, which contains all the elements of infinite order of $G_\infty(K,f,t)$. Then, we have $H_v=H$ for every $v\in T$, i.e. $H$ is a projection-invariant open subgroup of $G_\infty(K,f,t)$. Then, it follows from \cite[Theorem 17]{Ophelia} that $f$ must be a twisted Chebyshev polynomial.
\end{proof}

\subsection{Fixed-point proportion}

We now compute the fixed-point proportion of the geometric iterated Galois groups of polynomials with euclidean orbifold of type $(2,2,\infty)$.

\begin{proposition}
\label{proposition: fpp of 22}
Let $K$ be a separably closed field and $f\in K[x]$ be a polynomial of degree $d\ge 2$ with euclidean orbifold of type $(2,2,\infty)$. Then
$$\mathrm{FPP}(G_\infty(K,f,t))= \left \{ \begin{matrix} 
1/2 & \mbox{if $d$ is odd,} \\ 
1/4 & \mbox{if $d$ is even.}
\end{matrix}\right.$$
\end{proposition}
\begin{proof}
The proof is essentially the same as for iterated monodromy groups of complex Chebyshev polynomials of Jones in \cite[Proposition 1.2]{JonesAMS}. Indeed, by the proof of \cref{lemma: euclidean implies dihedral}, the group $G_\infty(K,f,t)$ has a dense dihedral subgroup $\langle g_{p_1},g_{p_2}\rangle$. Then, \cite[Proposition 8.5]{JonesAMS} implies that
$$\mathrm{FPP}(G_\infty(K,f,t)) = r/4,$$
where $r$ denotes the number of generators in $\{g_{p_1},g_{p_2}\}$ fixing an end in $\partial T$. 

If $d$ is odd, the action of each $g_{p_i}$ on the first level is the product of $(d-1)/2$ transpositions by \cref{proposition: action inertia generators}. By \cref{remark: sections of generators}, each element $g_{p_i}$ is conjugate to an element $g_i$ whose section at the only fixed point on the first level is $g_q$, where $q$ is either $p_1$ and $p_2$. Proceeding in this fashion, we conclude that both $g_{p_1}$ or $g_{p_2}$ fix an end in $\partial T$ and therefore $r = 2$. 

If $d$ is even, relabeling the post-critical points if necessary, we have $f^{-1}(p_1) \subseteq C_f$ and $\set{p_1,p_2} \subseteq f^{-1}(p_2)$. Therefore, the action on the first level of $g_{p_1}$ is the product of $d/2$ transpositions and $g_{p_1}$ has no fixed ends in $\partial T$. On the other hand, $g_{p_2}$ is conjugate to an element $g$ such that $x^g = x$ and $g|_x = g_{p_2}$ for some $x \in \mathcal{L}_1$. This implies that $g_{p_2}$ is the only generator fixing an end in $\partial T$ and consequently $r = 1$.
\end{proof}

Now, we are in position to prove \cref{Theorem: main IMG intro}, namely only the polynomials with euclidean orbifolds of type $(2,2,\infty)$ yield positive fixed-point proportion:

\begin{proof}[Proof of \cref{Theorem: main IMG intro}]
    By \cref{lemma: almost exceptional set}, either $\#\Upsilon_f\cap K=2$ and $\Upsilon_f\cap K=\Sigma_f$ or $\#\Upsilon_f\le 1$. In the former, we may combine  \cref{lemma: euclidean and exceptional} and \cref{proposition: fpp of 22} to obtain that 
    $$\mathrm{FPP}(G_\infty(K,f,t))= \left \{ \begin{matrix} 
1/2 & \mbox{if $d$ is odd,} \\ 
1/4 & \mbox{if $d$ is even.}
\end{matrix}\right.$$
    
In the latter, we get $G_\infty(K^{\mathrm{sep}},f,t)$ is mixing by \cref{lemma: mixing polynomial}. Then, the result follows from \cref{Theorem: main result FPP intro}.
\end{proof}



\bibliographystyle{unsrt}

\end{document}